\documentclass{amsart}
\usepackage{amsmath}
\usepackage{amsfonts}
\usepackage{amsthm}
\usepackage{amssymb}
\usepackage[dvipdfm]{graphicx}
\usepackage{amscd}

\title{On behavior of pairs of Teichm\"{u}ller geodesic rays}
\author{Masanori Amano}
\address{Department of Mathematics Tokyo Institute of Technology 2-12-1 Ookayama, Meguroku,
Tokyo 152-8551, JAPAN}
\email{amano.m.ab@m.titech.ac.jp}
\subjclass[2000]{Primary~32G15, Secondary~30F60}
\keywords{Teichm\"{u}ller space; Teichm\"{u}ller distance; Teichm\"{u}ller geodesic; Augmented Teichm\"{u}ller space; Gardiner-Masur compactification; Horofunction}

\theoremstyle{plain}
\newtheorem{thm}{Theorem}[section]

\theoremstyle{definition}
\newtheorem{defi}[thm]{Definition}

\theoremstyle{plain}
\newtheorem{prop}[thm]{Proposition}

\theoremstyle{plain}
\newtheorem{lemma}[thm]{Lemma}

\theoremstyle{plain}
\newtheorem{cor}[thm]{Corollary}

\theoremstyle{remark}

\theoremstyle{remark}
\newtheorem*{rem}{\bf Remark}

\theoremstyle{definition}

\DeclareMathOperator{\ext}{Ext}

\DeclareMathOperator{\jac}{Jac}

\begin{document}

\begin{abstract}
In this paper, we obtain the explicit limit value of the Teichm\"{u}ller distance between two Teichm\"{u}ller geodesic rays which are determined by Jenkins-Strebel differentials having a common end point on the augmented Teichm\"{u}ller space.
Furthermore, we also obtain a condition under which these two rays are asymptotic.
This is similar to a result of Farb and Masur.
\end{abstract}

\maketitle

\section{Introduction}
Let $T(X)$ be the Teichm\"{u}ller space of an analytic finite Riemann surface $X$.
Each Teichm\"{u}ller geodesic ray on $T(X)$ is determined by a holomorphic quadratic differential on an initial point.
We are interested in the behavior of two geodesic rays near the boundary of the Teichm\"{u}ller space.

An interesting question is when the Teichm\"{u}ller distance between the given two Teichm\"{u}ller geodesic rays are bounded, or diverge.
This question is answered completely by Ivanov \cite{Ivanov01}, Lenzhen and Masur \cite{LenMas10} and Masur \cite{Masur75}, \cite{Masur80}.
The details are following.
First, Masur showed that if the measured foliations of the given rays are Jenkins-Strebel and topologically equivalent, then the rays are bounded (\cite{Masur75}).
Masur also showed that if the measured foliations are uniquely ergodic and topologically equivalent with the condition that the sets of their critical trajectories do not contain closed loops, then the rays are asymptotic (and so bounded) (\cite{Masur80}).
Ivanov showed that if the measured foliations are absolutely continuous, then the rays are bounded, and if the measured foliations have non-zero intersection number, then the rays are divergent (\cite{Ivanov01}).
Finally, Lenzhen and Masur showed that if the measured foliations are not absolutely continuous, or the measured foliations are not topologically equivalent and have zero intersection number, then the rays are divergent (\cite{LenMas10}).

On the other hand, Farb and Masur \cite{FarMas10} considered two Jenkins-Strebel rays in the moduli space given by modularly equivalent holomorphic quadratic differentials.
They showed that under some conditions of initial points of two rays, the limit value of the distance between the rays in the moduli space is the distance between the end points of the rays in the boundary of the moduli space.
If in addition, the end points of the given rays are coincide, then they are asymptotic.

In this paper, we consider two Teichm\"{u}ller geodesic rays $r$, $r'$ on the Teichm\"{u}ller space $T(X)$ starting at $[Y,f]$, $[Y',f']$ with the conditions that their holomorphic quadratic differentials $q$, $q'$ are Jenkins-Strebel and the measured foliations $f^{-1}_*(H(q))$, $f'^{-1}_*(H(q'))$ are topologically equivalent.
Furthermore, we assume that the end points of $r$, $r'$ on the augmented Teichm\"{u}ller space are coincide.
Under the above assumption, we determine the limit value of the Teichm\"{u}ller distance between two points $r(t)$, $r'(t)$ on the given rays $r$, $r'$ respectively (Theorem \ref{main}).
Furthermore, we show that if two rays are modularly equivalent, then they are asymptotic (Corollary \ref{cor}).
This is similar to a theorem of Farb and Masur \cite{FarMas10}.
In Theorem \ref{main}, the limit value depends on the modulus of each annulus which is determined by the holomorphic quadratic differentials on the initial points of given rays.
Therefore, this value depends on the choice of the initial points.
We consider the minimum of the limit value of the Teichm\"{u}ller distance between two rays when we shift the initial points of given rays.
We show that the minimum is represented by the detour metric between the end points on the Gardiner-Masur boundary of $T(X)$ (Proposition \ref{prop}).

\section{Background}

\subsection{Teichm\"{u}ller spaces}
Let $X$ be a Riemann surface of type $(g,n)$ with $3g-3+n>0$.
The {\it Teichm\"{u}ller space} $T(X)$ of $X$ is the set of all equivalence classes $[Y,f]$ of pairs of a Riemann surface $Y$ and a quasiconformal mapping $f:X\rightarrow Y$.
Two pairs $(Y_1,f_1)$ and $(Y_2,f_2)$ are equivalent if there is a conformal mapping $h:Y_1\rightarrow Y_2$ such that $h\circ f_1$ is homotopic to $f_2$.
The Teichm\"{u}ller space $T(X)$ has a complete distance, called the {\it Teichm\"{u}ller distance} $d_{T(X)}$.
For any $p_1=[Y_1,f_1]$, $p_2=[Y_2,f_2]\in T(X)$, the distance is defined by
	\begin{center}
	$\displaystyle d_{T(X)}(p_1,p_2)=\frac{1}{2}\inf _h\log K(h)$,
	\end{center}
where $h$ ranges over all quasiconformal mappings $h:Y_1\rightarrow Y_2$ such that $h\circ f_1$ is homotopic to $f_2$, and $K(h)$ means the maximal quasiconformal dilatation of $h$.

\subsection{Holomorphic quadratic differentials}
In this section, we refer to \cite{Strebel84} in detail.

A {\it holomorphic quadratic differential} $q$ on $X$ is a tensor which is represented locally by $q=q(z)dz^2$ where $q(z)$ is a holomorphic function of the local coordinate $z=x+iy$ on $X$.
We allow holomorphic quadratic differentials to have simple poles at the punctures of $X$.
For each holomorphic quadratic differential $q=q(z)dz^2$, $|q|$ is defined locally by the differential 2-form $|q(z)|dxdy$ on each coordinate $z$.
The norm of $q$ is defined by $\|q\|=\iint _X|q(z)|dxdy$.
We treat holomorphic quadratic differentials whose norm are finite.
A holomorphic quadratic differential $q$ is called of {\it unit norm} if it satisfies $\|q\|=1$.
A {\it critical point} of $q$ is a zero of $q$ or a puncture of $X$.
Each non-zero holomorphic quadratic differential has finitely many critical points.
In a neighborhood of non-critical points, there exists a local coordinate $\zeta $ on $X$ such that $q=d\zeta ^2$.
Indeed, let $p_0$ be a non-critical point, then $q^{\frac{1}{2}}=q(z)^{\frac{1}{2}}dz$ has a single valued holomorphic branch in some neighborhood $U$ around $p_0$.
The new coordinate
	\begin{center}
	$\displaystyle \zeta (p)=\int ^p_{p_0}q^\frac{1}{2}$
	\end{center}
is defined for any $p\in U$, and this coordinate satisfies $q=d\zeta ^2$.
We call the coordinate $\zeta $ a {\it $q$-coordinate} on $X$.
For any two $q$-coordinates $\zeta _1$, $\zeta _2$ in a common neighborhood $U$, the equation $\zeta _2=\pm \zeta _1+c$ where $c\in \mathbb C$ holds, because $q=d\zeta _1^2=d\zeta _2^2$.
A smooth path $z=\gamma (t)$ on $X$ is called a {\it horizontal trajectory} of $q$ if it is a maximal path which satisfies $q(\gamma (t))\dot \gamma (t)^2>0$.
We notice that a horizontal trajectory of $q$ consists of non-critical points of $q$.
A horizontal trajectory of $q$ is represented by a Euclidean horizontal segment in $q$-coordinates.
All horizontal trajectories of $q$ are classified by the following three types.
A horizontal trajectory $\gamma $ of $q$ is {\it critical} if it joins critical points of $q$, {\it closed} if it is a closed path, {\it recurrent} otherwise.
The recurrent trajectory is dense on a subsurface of $X$ which is surrounded by critical trajectories.
Let $\Gamma _q$ be the set of all critical points and critical trajectories of $q$.
A component of $X-\Gamma _q$ is an {\it annulus} which is swept out by closed trajectories of $q$ such that they are homotopic to each other, or a {\it minimal domain} which consists of infinitely many recurrent trajectories of $q$.
If all components of $X-\Gamma _q$ are annuli, we call $q$ a {\it Jenkins-Strebel differential}.

\subsection{Measured foliations}
A {\it foliation} $F=\{(U_j, z_j)\}_j$ {\it with singularities} on $X$ is given by the pair of an open covering $\{ U_j\}_j$ of $X-\{ $finite distinguished points$\}$ and each local coordinate $z_j$ on $U_j$ such that the equation $z_k=\pm z_j+c$ where $c\in \mathbb C$ holds if $U_j\cap U_k\not=\emptyset$.
These distinguished points and punctures of $X$ are called {\it singularities} of $F$.
They are $p$-pronged singularities for $p\geq 1$, but the case of $p=1,2$ is attained only at the punctures of $X$.
A maximal horizontal segment with respect to $\{ z_j\}_j$ is called a {\it leaf} of $F$.
A {\it transverse measure} $\mu $ of $F$ is given by a measure on the set of all transversal arcs of leaves of $F$ such that if transversal arcs $\alpha $ and $\beta $ are moved to each other by the homotopy and each orbit of which is contained in a single leaf of $F$, then $\mu (\alpha )=\mu (\beta )$.
The pair $(F,\mu )$ is called a {\it measured foliation} on $X$.
Let $\mathcal S$ be the set of all homotopy classes of non-trivial and non-peripheral simple closed curves on $X$.
For any measured foliation $(F,\mu )$ and any $\alpha \in \mathcal S$, we can define the {\it intersection number}
	\begin{center}
	$\displaystyle i((F,\mu ),\alpha )=\inf _{\alpha ' \in \alpha }\int _{\alpha '}d\mu $,
	\end{center}
where $\alpha '$ ranges over all simple closed curves in $\alpha $.
Two pairs $(F_1, \mu _1)$ and $(F_2, \mu _2)$ are equivalent if the equation
	\begin{center}
	$i((F_1, \mu _1), \alpha )=i((F_2, \mu _2), \alpha )$
	\end{center}
holds for any $\alpha \in \mathcal S$.
We denote by $\mathcal {MF}(X)$ the set of all equivalence classes of measured foliations on $X$.
The set $\mathcal {MF}(X)$ has the week-topology which is induced by intersection number functions in $\mathbb R_{\geq 0}^{\mathcal S}$.
We denote by $[F,\mu ]$ the equivalence class of $(F,\mu )$.
We consider the space of measured foliations $\mathcal {MF}(Y)$ on any other Riemann surface $Y$ of the same type as $X$.
For any homeomorphism $f:X\rightarrow Y$, there exists a homeomorphism $f_*:\mathcal {MF}(X)\rightarrow \mathcal {MF}(Y)$ which is defined by $f_*([F,\mu ])=[f(F), \mu \circ f^{-1}]\in \mathcal {MF}(Y)$ for any $[F,\mu ]\in \mathcal {MF}(X)$.
After this, we denote by $\mu $ as the equivalence class of a measured foliation $[F,\mu ]\in \mathcal {MF}(X)$.
We can see that $\mathbb R _{\geq 0}\times \mathcal S\subset \mathcal {MF}(X)$.
For any $b\geq 0$ and any $\gamma \in \mathcal S$, the measured foliation $b\gamma $ consists of closed leaves which are homotopic to $\gamma $ and the leaves tend to singular points, and these closed leaves produce an annulus of height $b$.
This measured foliation satisfies $i(b\gamma, \alpha)=bi(\gamma, \alpha)$ for any $\alpha \in \mathcal S$.
The right side of this equation is a geometric intersection number of simple closed curves.
Thurston showed that $\overline {\mathbb R_{\geq 0}\times \mathcal S}=\mathcal {MF}(X)$. (For instance, we refer to \cite{FatLauPoe79}.)
Then, we can see that intersection number functions are defined on $\mathcal {MF}(X)\times \mathcal {MF}(X)$. (We refer to \cite{Rees81} for details.)
We notice that for any homeomorphism $f:X\rightarrow Y$ and any $\mu ,\nu \in \mathcal {MF}(X)$, the equation $i(f_*(\mu ),f_*(\nu ))=i(\mu ,\nu )$ holds.

\begin{rem}
The set $\mathcal {MF}(X)$ contains the zero-measured foliation, denoted by $0$, that is, it has the zero intersection number with all measured foliations.
Of course, for any $\mu \in \mathcal {MF}(X)$, we regard the measured foliation $0\mu $ as $0$.
\end{rem}

\subsection{Measured foliations and holomorphic quadratic differentials}
For any non-zero holomorphic quadratic differential $q$ on $X$, we can define the measured foliation $H(q)\in \mathcal {MF}(X)$ consists of all horizontal trajectories of $q$ as leaves and $|dy|$ as a transverse measure where $z=x+iy$ is the $q$-coordinate.
The singularities of $H(q)$ are the critical points of $q$.

\begin{rem}
There is a one-to-one correspondence between the set of all holomorphic quadratic differentials of finite norm on $X$ and $\mathcal {MF}(X)$, we refer to \cite{HubMas79}.
\end{rem}

If $q$ has an annulus $A$ in $X-\Gamma _q$, then the restriction to $A$ of the measured foliation $H(q)$ can be written as $H(q)|_A=b\gamma $ where $\gamma \in \mathcal S$ corresponds to closed trajectories which sweep out in $A$, and $b>0$ is the height of $A$ with respect to the metric $|dy|$.
If $q$ has a minimal domain $M$ in $X-\Gamma _q$, then the restriction to $M$ of the measured foliation $H(q)$ can be written as $H(q)|_M=\sum ^{p}_{i=1}b_ie_i$ where $p>0$ is bounded by the number which is determined by the topology of $M$, $b_i\geq 0$ for any $i=1,\cdots ,p$ and $\{e_i \}_{i=1}^p$ is the set of ergodic transverse measures which are projectively-distinct and pairwise having zero intersection number (that is, for any $i\not =i'$ and any $k\geq 0$, $e_{i'}\not =ke_i$ and $i(e_i,e_{i'})=0$).
A transverse measure $e$ is called an {\it ergodic transverse measure} if it is non-zero and cannot be written as a sum of projectively-distinct and non-zero measured foliations.
(For ergodicity, we refer to \cite{KatHas95} for more informations.)
The holomorphic quadratic differential $q$ has finitely many critical points, so the measured foliation $H(q)$ has finitely many these domains.
Therefore, the measured foliation $H(q)$ can be written as
	\begin{center}
	$\displaystyle H(q)=\sum _{j=1}^kb_jG_j$,
	\end{center}
where $G_j$ is (a homotopy class of) a simple closed curve or an ergodic measure for any $j=1,\cdots ,k$ such that these are projectively-distinct and pairwise having zero intersection number, and $b_j>0$ if $G_j\in \mathcal S$, $b_j\geq 0$ if $G_j$ is an ergodic measure. (For this notation, we refer the reader to \cite{Ivanov92} for more details.)
In particular, if $q$ is a Jenkins-Strebel differential, then we can write
	\begin{center}
	$\displaystyle H(q)=\sum _{j=1}^kb_j\gamma _j$,
	\end{center}
where $b_1,\cdots ,b_k$ are positive real numbers and $\gamma _1,\cdots ,\gamma _k$ are distinct simple closed curves such that $i(\gamma _j,\gamma _{j'})=0$ for any $j\not =j'$, and in this situation, we also call that $H(q)$ is {\it Jenkins-Strebel}.
If $X-\Gamma _q$ has only one minimal domain and the measured foliation $H(q)$ is represented by $be$ where $b>0$ and $e$ is an ergodic measure, and any other topologically equivalent measured foliation $\mu $ on $X$ is represented by $\mu =b'e$ where $b'>0$, then $q$ and $H(q)$ are called {\it uniquely ergodic}.
We come back to for the general case and see that
	\begin{center}
	$\displaystyle i\left(\sum _{j=1}^kb_jG_j,\alpha \right)=\sum _{j=1}^kb_ji(G_j,\alpha )$
	\end{center}
for any $\alpha \in \mathcal S$.
For any $j=1,\cdots ,k$, we set
	\begin{center}
	$\displaystyle m_j=\frac{b_j}{i(G_j,V(q))}$,
	\end{center}
where the measured foliation $V(q)$ is defined by $H(-q)$.
If $G_j$ is a simple closed curve $\gamma _j\in \mathcal S$, then $a_j:=i(\gamma _j,V(q))$ means the infimum of the horizontal lengths of the simple closed curves in $\gamma _j$ with respect to the metric $|dx|$, and $m_j=\frac{b_j}{a_j}$ means a modulus of the annulus which is generated by $\gamma _j$, that is, the ratio of the height and the circumference of the annulus.

Let $q$, $q'$ be unit norm holomorphic quadratic differentials on $X$ and $H(q)$, $H(q')$ be measured foliations in $\mathcal {MF}(X)$ constructed by $q$, $q'$ respectively.

\begin{defi}
The pair of holomorphic quadratic differentials $q$, $q'$ or measured foliations $H(q)$, $H(q')$ is called {\it topologically equivalent} if there is a homeomorphism $\alpha :X-\Gamma _q\rightarrow X-\Gamma _{q'}$ which is homotopic to the identity such that the leaves of $H(q)$ are mapped to the leaves of $H(q')$.
In this situation, we can write their measured foliations as $H(q)=\sum _{j=1}^kb_jG_j$, $H(q')=\sum _{j=1}^kb'_jG_j$ where $G_j$ is a simple closed curve or an ergodic measure for any $j=1,\cdots ,k$ such that these are projectively-distinct and pairwise having zero intersection number, and $b_j$, $b'_j>0$ if $G_j\in \mathcal S$, $b_j$, $b'_j\geq 0$ if $G_j$ is an ergodic measure.
The pair of holomorphic quadratic differentials $q$, $q'$ or measured foliations $H(q)$, $H(q')$ is called {\it absolutely continuous} if $q$, $q'$ are topologically equivalent and for $H(q)=\sum _{j=1}^kb_jG_j$, $H(q')=\sum _{j=1}^kb'_jG_j$, the set of subscripts of non-zero coefficients $b_j$ of $H(q)$ and the one of $H(q')$ are coincide.
In this situation, we can regard as $b_j$, $b'_j>0$ for any $j=1,\cdots ,k$.
\end{defi}

\begin{rem}
If two holomorphic quadratic differentials $q$, $q'$ are Jenkins-Strebel and topologically equivalent, then they are absolutely continuous.
\end{rem}

\subsection{Teichm\"{u}ller geodesic rays}
Let $q$ be a unit norm holomorphic quadratic differential on $X$.
A quasiconformal mapping $f$ on $X$ is called a {\it Teichm\"{u}ller mapping} for $q$ if $f$ has the Beltrami coefficient $\frac{K(f)-1}{K(f)+1}\frac{\bar q}{|q|}$.
The existence and uniqueness for Teichm\"{u}ller mappings are the followings.

\begin{thm}{\rm (Teichm\"{u}ller's existence theorem)}
For any quasiconformal mapping $g:X\rightarrow Y$, there is a Teichm\"{u}ller mapping $f$ which is homotopic to $g$.
\end{thm}

\begin{thm}{\rm (Teichm\"{u}ller's uniqueness theorem)}
For any quasiconformal mapping $g:X\rightarrow Y$ which is homotopic to the Teichm\"{u}ller mapping $f$, the inequality $K(f)\leq K(g)$ holds and the equality holds if and only if $f=g$.
\end{thm}

These facts are called the Teichm\"{u}ller's theorem. (We refer to \cite{ImaTan92} for details.)
Therefore, a Teichm\"{u}ller mapping is attained the supremum of the definition of the Teichm\"{u}ller distance, i.e., for any $p_1=[Y_1,f_1]$, $p_2=[Y_2,f_2]$, there is the Teichm\"{u}ller mapping $h:Y_1\rightarrow Y_2$ which is homotopic to $f_2\circ f_1^{-1}$ such that $d_{T(X)}(p_1,p_2)=\frac{1}{2}\log K(h)$.

For any $p=[Y, f]\in T(X)$, let $q$ be a unit norm holomorphic quadratic differential on $Y$ and $z=x+iy$ be a $q$-coordinate.
For any $t\geq 0$, let $Y_t$ be a Riemann surface determined by the local coordinate $z_t=e^{-t}x+ie^ty$ and $g_t:Y\rightarrow Y_t$ be the Teichm\"{u}ller mapping which is determined by $z\mapsto z_t$.
We assume that $Y_0=Y$, $g_0=id_Y$.
We set $r(t)=[Y_t, g_t\circ f]$.
By the Teichm\"{u}ller's theorem, this mapping $r:[0,\infty )\rightarrow T(X)$ is the {\it Teichm\"{u}ller geodesic ray} on $T(X)$ starting at $p$ and having the holomorphic quadratic differential $q$.
If the holomorphic quadratic differential $q$ is of Jenkins-Strebel, the ray $r$ is called a {\it Jenkins-Strebel ray}.

\begin{defi}
Let $p=[Y,f]$, $p'=[Y',f']\in T(X)$, $q$, $q'$ be unit norm holomorphic quadratic differentials on $Y$, $Y'$ and $r$, $r'$ be Teichm\"{u}ller geodesic rays on $T(X)$ starting at $p$, $p'$ and having $q$, $q'$ respectively.
We denote by $H(q)\in \mathcal {MF}(Y)$, $H(q')\in \mathcal {MF}(Y')$ the measured foliations corresponding to $q$, $q'$ respectively.
We suppose that $f^{-1}_*(H(q))$, $f'^{-1}_*(H(q'))\in \mathcal {MF}(X)$ are absolutely continuous, then the measured foliations are written as $f^{-1}_*(H(q))=\sum _{j=1}^kb_jG_j$, $f'^{-1}_*(H(q'))=\sum _{j=1}^kb'_jG_j$, $H(q)=\sum _{j=1}^kb_jf_*(G_j)$ and $H(q')=\sum _{j=1}^kb'_jf'_*(G_j)$ where $b_1,\cdots ,b_k$, $b'_1,\cdots ,b'_k$ are positive real numbers and $G_j$ is a simple closed curve or an ergodic measure for any $j=1,\cdots ,k$ such that these are projectively-distinct and pairwise having zero intersection number.
We set $m_j=\frac{b_j}{i(f_*(G_j),V(q))}$, $m'_j=\frac{b'_j}{i(f'_*(G_j),V(q'))}$ for any $j=1,\cdots ,k$.
In this situation, the given rays $r$, $r'$ are called {\it modularly equivalent} if there is $\lambda >0$ such that $m'_j=\lambda m_j$ for any $j=1,\cdots ,k$.
\end{defi}

\begin{defi}
The pair of Teichm\"{u}ller geodesic rays $r$, $r'$ on $T(X)$ are called {\it bounded} if there is $M>0$ such that $d_{T(X)}(r(t),r'(t))<M$ for any $t\geq 0$, {\it divergent} if $d_{T(X)}(r(t),r'(t))\rightarrow +\infty $ as $t\rightarrow \infty $, and {\it asymptotic} if there is a choice of initial points $r(0)$, $r'(0)$ such that $d_{T(X)}(r(t),r'(t))\rightarrow 0$ as $t\rightarrow \infty $, in other words, for the given rays $r(t)$, $r'(t)$, there is $\alpha \in \mathbb R$ such that $d_{T(X)}(r(t),r'(t+\alpha ))\rightarrow 0$ as $t\rightarrow \infty $.
\end{defi}

\subsection{The end point of a Jenkins-Strebel ray} \label{end}
In this section, we refer to \cite{Abikoff77}, \cite{HerSch07} and \cite{ImaTan92}.

\begin{defi}(Riemann surfaces with nodes)
A connected Hausdorff space $R$ is called a {\it Riemann surface of type $(g,n)$ with nodes} if $R$ satisfies the following three conditions:
	\begin{enumerate}
	\item Any $p\in R$ has a neighborhood which is homeomorphic to the unit disk $\mathbb D$ or the set $\{ (z_1,z_2)\in \mathbb C^2\ |\ |z_1|<1,|z_2|<1,z_1\cdot z_2=0\}$.
	(In the latter case, $p$ is called a {\it node} of $R$. We allow $R$ to have finitely many nodes.)
	\item Let $p_1,\cdots,p_k$ be nodes of $R$.
	We denote by $R_1,\cdots,R_r$ the connected components of $R-\{ p_1,\cdots,p_k\}$.
	For $i=1,\cdots,r$, each $R_i$ is of type $(g_i,n_i)$ which satisfies $2g_i-2+n_i>0$, $n=\sum _{i=1}^rn_i-2k$, and $g=\sum _{i=1}^rg_i-r+k+1$.
	\end{enumerate}
\end{defi}

\begin{rem}
The condition (2) means that we get a Riemann surface of type $(g,n)$ without nodes by opening each node of $R$.
All Riemann surfaces of type $(g,n)$ without nodes are included to this definition.
\end{rem}

\begin{defi}(Augmented Teichm\"{u}ller spaces)
Let $X$ be a Riemann surface of type $(g,n)$ without nodes which satisfies $3g-3+n>0$.
The {\it augmented Teichm\"{u}ller space} $\hat T(X)$ of $X$ is the set of all equivalence classes $[R,f]$ of pairs of a Riemann surface $R$ of type $(g,n)$ with nodes and a deformation $f:X\rightarrow R$ which is a mapping such that it contracts some disjoint loops on $X$ to points (the nodes of $R$) and is a homeomorphism except their loops on $X$.
Two pairs $(R_1,f_1)$ and $(R_2,f_2)$ are equivalent if there is a biholomorphic mapping $h:R_1\rightarrow R_2$ such that $f_2$ is homotopic to $h\circ f_1$.
Here, for Riemann surfaces with nodes $R$ and $S$, a homeomorphism $f:R\rightarrow S$ is called biholomorphic if each restricted mapping of $f$ which maps a component of $R-\{$nodes of $R\}$ onto a component of $S-\{$nodes of $S\}$ is biholomorphic.
A topology on $\hat T(X)$ is defined by the following neighborhoods.
For any compact neighborhood $V$ of the set of nodes in $R$ and any $\varepsilon >0$, a neighborhood $U_{V,\varepsilon }$ of a point $[R,f]$ is defined by
	\begin{center}
	$U_{V,\varepsilon }=\{ [S,g]\in \hat T(X)$\ $|$\ there is a deformation $h:S\rightarrow R$ which is $(1+\varepsilon )$-quasiconformal on $h^{-1}(R-V)$ such that $f$ is homotopic to $h\circ g\}$.
	\end{center}
\end{defi}

Let $p=[Y, f]\in T(X)$ and $q$ be a unit norm Jenkins-Strebel differential on $Y$.
We denote by $r(t)=[Y_t,g_t\circ f]$ for any $t\geq 0$ the Jenkins-Strebel ray on $T(X)$ starting at $p$ and having $q$.
We consider the end point of $r$ in $\hat T(X)$.
We denote by $H(q)=\sum _{j=1}^kb_j\gamma _j$ the measured foliation constructed by $q$.
The surface $Y-\Gamma _q$ consists of annuli corresponding to $\gamma _1,\cdots ,\gamma_k$.
By the $q$-coordinate $z$, their annuli are represented by rectangles such that each rectangle identifies its vertical sides.
Then, we denote by
	\begin{center}
	$\displaystyle R_j(0)=\{ z\in \mathbb C\ |\ 0\leq$Re$z\leq a_j$, $0<$Im$z<b_j\}/(i$Im$z\sim a_j+i$Im$z)$
	\end{center}
the annulus corresponding to $\gamma _j$ where $a_j=i(\gamma _j,V(q))$ and $\sim $ means the identification of vertical sides of the rectangle for any $j=1,\cdots ,k$.
Let $m_j=\frac{b_j}{a_j}$ be a modulus of $R_j(0)$ for any $j=1,\cdots ,k$.
Two horizontal sides of each $\overline {R_j(0)}$ are glued other ones by transformations of form $z\mapsto \pm z+c$ where $c\in \mathbb C$.
The critical points of $q$ are on the horizontal sides of each $\overline {R_j(0)}$.
The Riemann surface $Y$ is obtained by $\overline {R_1(0)},\cdots ,\overline {R_k(0)}$ with these gluing mappings.
Now, we view $R_1(0),\cdots ,R_k(0)$ as round annuli.
First, for any $j=1,\cdots ,k$, we cut $R_j(0)$ at the half height:
	\begin{center}
	$\displaystyle R_j^1(0)=\{ z\in \mathbb C\ |\ 0\leq $Re$z\leq a_j$, $0<$Im$z\leq \frac{b_j}{2}\}/(i$Im$z\sim a_j+i$Im$z)$,
	\end{center}
	\begin{center}
	$\displaystyle R_j^2(0)=\{ z\in \mathbb C\ |\ 0\leq$Re$z\leq a_j$, $\frac{b_j}{2}\leq$Im$z<b_j\}/(i$Im$z\sim a_j+i$Im$z)$.
	\end{center}
The corresponding round annuli are defined by the following:
	\begin{center}
	$\displaystyle A_j^1(0)=A_j^2(0)=\{ w\in \mathbb C\ |\ \exp(-m_j\pi )\leq |w|<1\}$.
	\end{center}
The half surface $R_j^1(0)$ is mapped to $A_j^1(0)$ by the mapping $z\mapsto w=\exp(2\pi i\frac{z}{a_j})$, $R_j^2(0)$ is mapped to $A_j^2(0)$ by the mapping $z\mapsto w=\exp(2\pi i\frac{a_j+ib_j-z}{a_j})$.
We glue $A_j^1(0)$ and $A_j^2(0)$ at the line of these inner boundary $|w|=\exp(-m_j\pi )$ by the mapping $w\mapsto \frac{\exp(-2m_j\pi )}{w}$ and we denote by $A_j(0)$ the resulting surface.
There is a natural biholomorphic mapping of $A_j(0)$ onto $R_j(0)=R_j^1(0)\cup R_j^2(0)$.
This mapping is extended continuously to the mapping $\overline{A_j(0)}$ onto $\overline{R_j(0)}$.
Then, we obtain the base surface $Y$ after gluing $\overline{A_1(0)},\cdots ,\overline{A_k(0)}$ by the gluing mappings essentially the same as in the case of $\overline{R_1(0)},\cdots ,\overline{R_k(0)}$.
We recall that the Teichm\"{u}ller mapping $g_t:Y\rightarrow Y_t$ is represented by $z\mapsto z_t=e^{-t}x+ie^ty$ for any $t\geq 0$.
We fix any $t\geq 0$.
We denote by $R_1(t),\cdots,R_k(t)$ the surfaces which are transformed from $R_1(0),\cdots ,R_k(0)$ by the mapping $g_t$.
For any $j=1,\cdots ,k$, these are expressed by
	\begin{center}
	$\displaystyle R_j(t)=\{ z_t\in \mathbb C\ |\ 0\leq$Re$z_t\leq e^{-t}a_j$, $0<$Im$z_t<e^tb_j\}/(i$Im$z_t\sim e^{-t}a_j+i$Im$z_t)$.
	\end{center}
Similarly, we define
	\begin{center}
	$\displaystyle R_j^1(t)=\{ z_t\in \mathbb C\ |\ 0\leq $Re$z_t\leq e^{-t}a_j$, $0<$Im$z_t\leq \frac{e^tb_j}{2}\}/(i$Im$z_t\sim e^{-t}a_j+i$Im$z_t)$,
	\end{center}
	\begin{center}
	$\displaystyle R_j^2(t)=\{ z_t\in \mathbb C\ |\ 0\leq$Re$z_t\leq e^{-t}a_j$, $\frac{e^tb_j}{2}\leq$Im$z_t<e^tb_j\}/(i$Im$z_t\sim e^{-t}a_j+i$Im$z_t)$,
	\end{center}
	\begin{center}
	$\displaystyle A_j^1(t)=A_j^2(t)=\{ w_t\in \mathbb C\ |\ \exp(-e^{2t}m_j\pi )\leq |w_t|<1\}$.
	\end{center}
The mapping of $R_j^1(t)$ onto $A_j^1(t)$ is obtained by $z_t\mapsto w_t=\exp(2\pi i\frac{e^tz}{a_j})$, $R_j^2(t)$ onto $A_j^2(t)$ is obtained by $z_t\mapsto w_t=\exp(2\pi i\frac{a_j+ie^{2t}b_j-e^tz}{a_j})$.
The gluing mapping of $A_j^1(t)$ onto $A_j^2(t)$ is $w_t\mapsto \frac{\exp(-2e^{2t}m_j\pi )}{w_t}$ in their inner boundary $|w_t|=\exp(-e^{2t}m_j\pi )$, and we denote by $A_j(t)$ the resulting surface.
There is a biholomorphic mapping of $A_j(t)$ onto $R_j(t)$ which can be extended continuously to their closures.
The surface $Y_t$ is obtained by $\overline {R_1(t)},\cdots ,\overline {R_k(t)}$ with the gluing mappings the same as in the case of $t=0$, that is, we use the coordinate $z_t$ instead of $z$.
The surface $Y_t$ is also obtained by $\overline{A_1(t)},\cdots ,\overline{A_k(t)}$ similarly.
The Teichm\"{u}ller mapping $g_t$ is considered as the mapping of $A_j^l(0)$ onto $A_j^l(t)$ by $w=re^{i\theta }\mapsto w_t=r^{\exp(2t)}e^{i\theta }$ for any $j=1,\cdots ,k$ and $l=1,2$.
Therefore, the diagram of Figure \ref{J-Sray} is commutative.

	\begin{figure}[!ht]
	\begin{center}
	\includegraphics[keepaspectratio, scale=0.70]
	{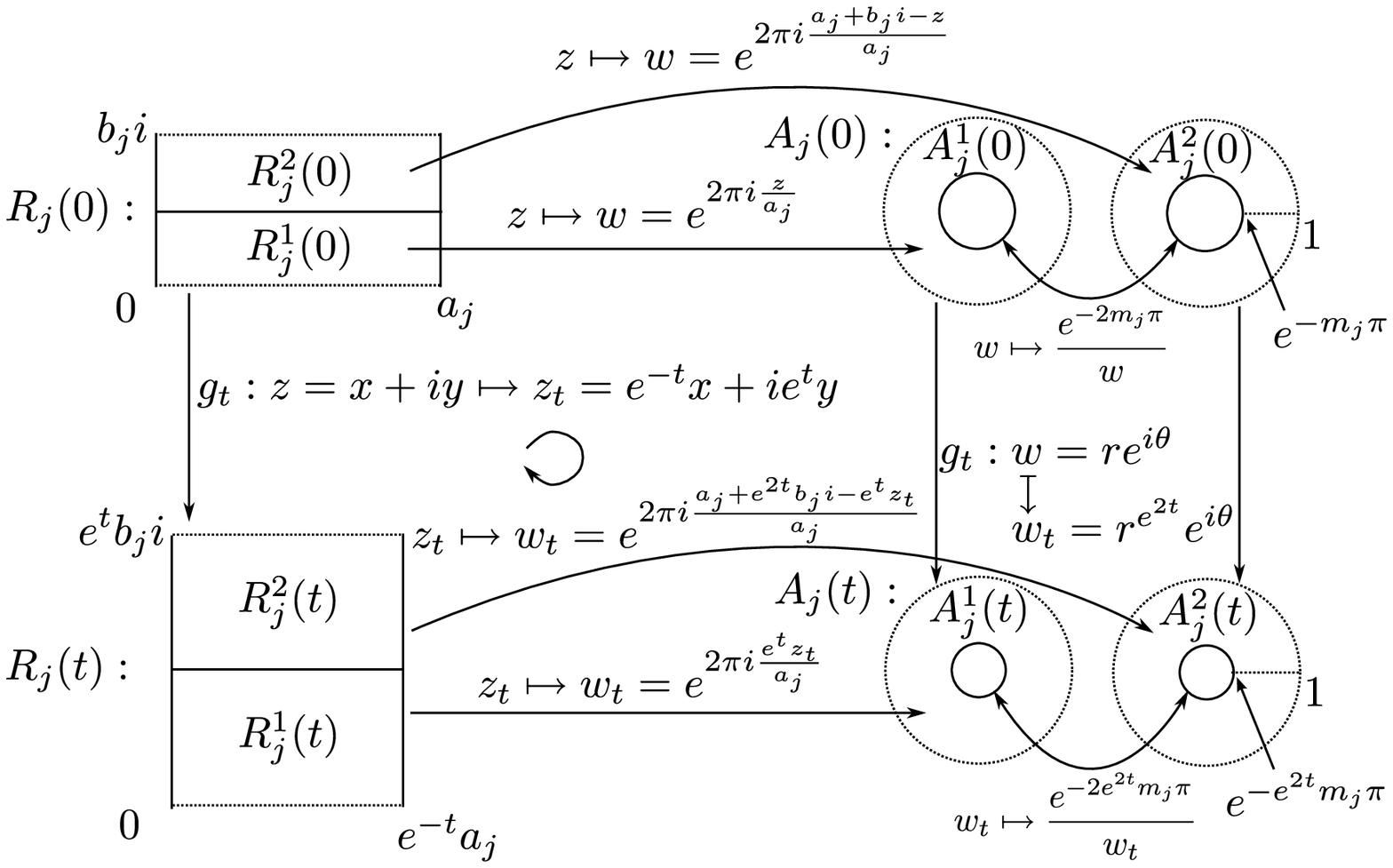}
	\caption{}
	\label{J-Sray}
	\end{center}
	\end{figure}

\begin{rem}
We notice that the moduli of $R_j(t)$ and $A_j(t)$ are equal to $e^{2t}m_j$ for any $t\geq 0$ and any $j=1,\cdots ,k$.
\end{rem}

We consider the limit of $A_j(t)$ as $t\rightarrow \infty$ for any $j=1,\cdots ,k$.
We set $A_j^1(\infty )=A_j^2(\infty )=\{ w_{\infty }\in \mathbb C\ |\ 0<|w_{\infty }|<1\}$ and $\{ pt\}$ is the set consists of an arbitrary point. The disjoint union $A_j(\infty )= A_j^1(\infty )\cup A_j^2(\infty )\cup \{ pt\}$ becomes a complex cone by the following chart:
	\begin{center}
	$\displaystyle \Phi :A_j(\infty )\rightarrow \{ (w_{\infty }^1,w_{\infty }^2)\in \mathbb C^2\ |\ |w_{\infty }^1|<1,|w_{\infty }^2|<1, w_{\infty }^1\cdot w_{\infty }^2=0\}$,
	\end{center}
	\begin{center}
	$\displaystyle \Phi |_{A_j^1(\infty )}(w_{\infty })=(w_{\infty },0)$, $\Phi |_{A_j^2(\infty )}(w_{\infty })=(0,w_{\infty })$, $\Phi (pt)=(0,0)$
	\end{center}
for any $j=1,\cdots ,k$.
We denote by $Y_{\infty }$ the surface constructed by $\overline{A_1(\infty )}, \cdots $, $\overline{A_k(\infty )}$ with the gluing mappings the same as in the case of $\overline{A_1(0)},\cdots ,\overline{A_k(0)}$ which are used the coordinate $w_{\infty }$ instead of $w$.
The mapping $g_{\infty }:Y\rightarrow Y_{\infty }$ is constructed by the mapping of $A_j^l(0)$ onto $A_j^l(\infty )\cup \{ pt\}$ by $w=re^{i\theta }\mapsto w_{\infty }=h_j(r)e^{i\theta }$ where $h_j:[\exp(-m_j\pi ),1)\rightarrow [0,1)$ is an arbitrary monotonously increasing diffeomorphism for any $j=1,\cdots ,k$ and $l=1,2$.
The homotopy class of $g_{\infty }$ is independent of the choices of $h_j$ for any $j=1,\cdots ,k$.
Then, we obtain $[Y_{\infty },g_{\infty }\circ f]$ in $\hat T(X)$ and denote it by $r(\infty )$.
From the definition of a neighborhood of $[Y_{\infty },g_{\infty }\circ f]$, the following proposition holds.

\begin{prop}{\rm (\cite{HerSch07})} \label{convergence}
The Jenkins-Strebel ray $r(t)=[Y_t,g_t\circ f]$ on $T(X)$ starting at $p=[Y,f]$ and having the unit norm Jenkins-Strebel differential $q$ on $Y$ converges to a point $r(\infty )=[Y_{\infty },g_{\infty }\circ f]$ in $\hat T(X)$.
\end{prop}

\begin{rem}
We can reconstruct the surface $Y_t$ from $Y_{\infty }$ for any $t\geq 0$.
First, for any $j=1,\cdots ,k$ and $l=1,2$, we remove the punctured disk $\{ w_{\infty }\in \mathbb C\ |\ 0<|w_{\infty }|<\exp(-e^{2t}m_j\pi )\}$ from $A_j^l(\infty )$ of $Y_{\infty }$.
Then, the interior of each resulting annulus is biholomorphic to the interior of $A_j^l(t)$.
We obtain $Y_t$ by gluing $\{ \overline{A_j^l(t)}\} _{j=1,\cdots ,k}^{l=1,2}$ with the suitable gluing mappings.
\end{rem}

\subsection{Extremal lengths}
Let $\rho $ be a locally $L^1$-measurable conformal metric on a Riemann surface $Y$, i.e., it is represented by the form $\rho =\rho (z)|dz|$ on any local coordinate $z$ of $Y$ where $\rho (z)$ is a non-negative measurable function of $z$ such that the equation $\rho (z)|dz|=\rho (w)|dw|$ holds for any local coordinates $z,w$ on a common neighborhood $U$ of $Y$.
For any $\gamma \in \mathcal S$, we define the $\rho $-length of $\gamma $ on $Y$ and the $\rho $-area of $Y$ by
	\begin{center}
	$\displaystyle l_\rho (\gamma )=\inf _{\gamma '\in \gamma }\int _{\gamma '}\rho (z)|dz|$,\\
	$\displaystyle A_{\rho }=\iint_Y\rho (z)^2dxdy$
	\end{center}
respectively.
The {\it extremal length} $\ext _Y(\gamma )$ of $\gamma $ on $Y$ is defined by the following:
	\begin{center}
	$\displaystyle \ext _Y(\gamma )=\sup _{\rho }\frac{l_\rho (\gamma )^2}{A_{\rho }}$,
	\end{center}
where $\rho $ ranges over all locally $L^1$-measurable conformal metrics on $Y$ such that $0<A_{\rho }<\infty $.

For any $p=[Y,f]\in T(X)$ and any $\alpha \in \mathcal S$ on $X$, we define
	\begin{center}
	$\ext _p(\alpha )=\ext _Y(f_*(\alpha ))$.
	\end{center}

\begin{thm}{\rm (\cite{Kerckhoff80})}
For any Riemann surface $Y$, each extremal length on $Y$ extends continuously on $\mathcal {MF}(Y)$ such that the equation $\ext _Y(t\nu )=t^2\ext _Y(\nu )$ holds for any $t\geq 0$ and any $\nu \in \mathcal {MF}(Y)$.
\end{thm}

Hence, for any $p=[Y,f]\in T(X)$ and any $\mu \in \mathcal {MF}(X)$, we define $\ext _p(\mu )=\ext _Y(f_*(\mu ))$ and the equation $\ext _p(t\mu )=t^2\ext _p(\mu )$ holds for any $t\geq 0$.
We notice that for any $p\in T(X)$ and any $\mu \in \mathcal {MF}(X)-\{ 0\}$, the extremal length $\ext _p(\mu )$ is positive.
We set $\mathcal {MF}(X)^*=\mathcal {MF}(X)-\{ 0\}$.

\begin{thm}{\rm (Kerckhoff's formula for the Teichm\"{u}ller distance \cite{Kerckhoff80})}
For any $p_1,p_2\in T(X)$, the Teichm\"{u}ller distance between $p_1$ and $p_2$ is represented by
	\begin{center}
	$\displaystyle d_T(p_1,p_2)=\frac{1}{2}\log \sup _{\mu \in \mathcal {MF}(X)^*}\frac{\ext _{p_2}(\mu )}{\ext _{p_1}(\mu )}$.
	\end{center}
\end{thm}

Kerckhoff's formula is useful for finding lower bounds of the Teichm\"{u}ller distance. 

\section{Proof of Theorem}
In this section, we prove Theorem \ref{main} which is our main theorem.

\begin{thm} \label{main}
Let $p=[Y,f]$, $p'=[Y',f']\in T(X)$, $q$, $q'$ be unit norm Jenkins-Strebel differentials on $Y$, $Y'$ and $r$, $r'$ be Jenkins-Strebel rays on $T(X)$ starting at $p$, $p'$ and having $q$, $q'$ respectively.
Let $r(\infty )$, $r'(\infty )$ be the end points of $r$, $r'$ on the augmented Teichm\"{u}ller space $\hat T(X)$ respectively.
We denote by $H(q)\in \mathcal {MF}(Y)$, $H(q')\in \mathcal {MF}(Y')$ the measured foliations corresponding to $q$, $q'$ respectively.
Suppose that the measured foliations $f^{-1}_*(H(q))$, $f'^{-1}_*(H(q'))\in \mathcal {MF}(X)$ are topologically equivalent.
In this situation, we can write $f^{-1}_*(H(q))=\sum _{j=1}^kb_j\gamma _j$, $f'^{-1}_*(H(q'))=\sum _{j=1}^kb'_j\gamma _j$, $H(q)=\sum _{j=1}^kb_jf_*(\gamma _j)$ and $H(q')=\sum _{j=1}^kb'_jf'_*(\gamma _j)$ where $b_1,\cdots ,b_k$, $b'_1,\cdots ,b'_k$ are positive real numbers and $\gamma _1,\cdots ,\gamma _k$ are distinct simple closed curves on $X$ such that $i(\gamma _j,\gamma _{j'})=0$ for any $j\not =j'$.
We denote by $a_j=i(f_*(\gamma _j),V(q))$, $a'_j=i(f'_*(\gamma _j),V(q'))$ the circumferences of the annuli of $H(q)$, $H(q')$ corresponding to $f_*(\gamma _j)$, $f'_*(\gamma _j)$, and $m_j=\frac{b_j}{a_j}$, $m'_j=\frac{b'_j}{a'_j}$ the corresponding moduli respectively, for any $j=1,\cdots ,k$.
If $r(\infty )=r'(\infty )$, then
	\begin{center}
	$\displaystyle \lim _{t\rightarrow \infty }d_{T(X)}(r(t),r'(t))=\frac{1}{2}\log \max _{j=1,\cdots ,k}\left\{ \frac{m'_j}{m_j},\frac{m_j}{m'_j}\right\}$.
	\end{center}
\end{thm}

We represent the Jenkins-Strebel rays by $r(t)=[Y_t,g_t\circ f]$, $r'(t)=[Y'_t,g'_t\circ f']$ for any $t\geq 0$, and their end points by $r(\infty )=[Y_{\infty },g_\infty \circ f]$, $r'(\infty )=[Y'_{\infty },g'_\infty \circ f']$ respectively.
Since the measured foliations $f^{-1}_*(H(q))$, $f'^{-1}_*(H(q'))$ are topologically equivalent, there is a homeomorphism $\alpha :X\rightarrow X$ which is homotopic to the identity such that the leaves of $f^{-1}_*(H(q))$ are mapped to the leaves of $f'^{-1}_*(H(q'))$.
Then, the mapping $f'\circ \alpha \circ f^{-1}$ which is homotopic to $f'\circ f^{-1}$ maps the leaves of $H(q)$ to the leaves of $H(q')$.
We denote by $A_j(0)$, $A'_j(0)$ the annuli corresponding to $f_*(\gamma _j)$, $f'_*(\gamma _j)$ respectively, for any $j=1,\cdots ,k$.
They are the same as in \S \ref{end}.
In this assumption, the annuli $A_j(0)$, $A'_j(0)$ have the moduli $m_j$, $m'_j$ respectively, for any $j=1,\cdots ,k$.
Then, the equations $f'_*\circ \alpha _*\circ f_*^{-1}(f_*(\gamma _j))=f'_*(\gamma _j)$, $f'\circ \alpha \circ f^{-1}(A_j(0))=A'_j(0)$ hold for any $j=1,\cdots ,k$.
For any $t\geq 0$ and any $j=1,\cdots ,k$, the annuli $A_j(t)$, $A'_j(t)$ can be determined the same as in \S \ref{end}, and we can see that $(g'_t\circ f')\circ \alpha \circ (g_t\circ f)^{-1}(A_j(t))=A'_j(t)$.

In order to prove the equation of Theorem \ref{main}, we consider the upper and lower estimates of the limit supremum and limit infimum of the distance between given rays respectively.

\subsection{Upper estimate}
First, we show that
	\begin{center}
	$\displaystyle \limsup _{t\rightarrow \infty }d_{T(X)}(r(t),r'(t))\leq \frac{1}{2}\log \max _{j=1,\cdots ,k}\left\{ \frac{m'_j}{m_j},\frac{m_j}{m'_j}\right\}$.
	\end{center}
The idea of the following lemma comes from \cite{Gupta11}.

\begin{lemma}
Let us choose $0<\varepsilon <1$ arbitrary.
Then, for any sufficiently large $t$, there is a quasiconformal mapping $F_t:Y_t\rightarrow Y'_t$ which is homotopic to $(g'_t\circ f')\circ (g_t\circ f)^{-1}$ such that the inequality $\displaystyle \lim _{t\rightarrow \infty }K(F_t)<\max _{j=1,\cdots ,k}\left\{ \frac{m'_j}{m_j},\frac{m_j}{m'_j}\right\}+\varepsilon $ holds.
\end{lemma}

\begin{proof}
We set $M_j=\frac{m'_j}{m_j}$ for any $j=1,\cdots ,k$.
By the equation $r(\infty )=r'(\infty )$, there exists a biholomorphic mapping $h:Y_{\infty }\rightarrow Y'_{\infty }$ such that $h\circ g_{\infty }\circ f$ is homotopic to $g'_{\infty }\circ f'$.
From \S \ref{end}, we can write
	\begin{center}
	$\displaystyle Y_{\infty }=\bigcup _{j=1}^k\overline{A_j^1(\infty )}\cup \overline{A_j^2(\infty )}$,\\
	$\displaystyle Y'_{\infty }=\bigcup _{j=1}^k\overline{A_j^{'1}(\infty )}\cup \overline{A_j^{'2}(\infty )}$,
	\end{center}
where $A_j^l(\infty )$, $A_j^{'l}(\infty )$ are the punctured disk $\mathbb D^*=\{ z\in \mathbb C\ |\ 0<|z|<1\}$ for any $j=1,\cdots ,k$ and $l=1,2$.
Now, we fix any $j=1,\cdots ,k$ and $l=1,2$.
We set $h_j^l=h|_{A_j^l(\infty )}: A_j^l(\infty )\rightarrow h(A_j^l(\infty ))\subset Y'_{\infty }$.
Since $h$ is a biholomorphic mapping, then we can set $h_j^l(0)=0$ and $\frac{dh_j^l(z)}{dz}\big|_{z=0}\not =0$.
We describe $h_j^l(z)=c_j^lz+c_{j,2}^lz^2+\cdots =c_j^lz+\psi _j^l(z)$ where $c_j^l\not =0$, $-\pi<\arg c_j^1\leq \pi$ and $-\pi\leq \arg c_j^2<\pi$.
For any $t\geq 0$, we set $\delta _j(t)=\exp(-e^{2t}m_j\pi)$, $\delta '_j(t)=\exp(-e^{2t}m'_j\pi)$, then $\delta '_j(t)=\delta _j(t)^{M_j}$.
After this, we assume that $A_j^l(t)=\mathbb D^*-\mathbb D_{\delta _j(t)}=\{ z\in \mathbb C\ |\ \delta _j(t)\leq |z|<1\}$ and $A_j^{'l}(t)=\mathbb D^*-\mathbb D_{\delta '_j(t)}=\{ z\in \mathbb C\ |\ \delta '_j(t)\leq |z|<1\}$ for any $t\geq 0$.
For sufficiently large $t$, we construct a quasiconformal mapping $F_{j,t}^l:\mathbb D^*-\mathbb D_{\delta _j(t)}\rightarrow h_j^l(\mathbb D^*)-\mathbb D_{\delta '_j(t)}$.
We consider the following three cases (1), (2) and (3).

(1) In the case of $M_j>1$, we take $X_j$ as
	\begin{center}
	$\displaystyle X_j<\frac{\log \frac{\varepsilon }{M_j+\varepsilon -1}}{\log M_j}<0$.
	\end{center}
This is equivalent to
	\begin{center}
	$\displaystyle M_j^{X_j}<\frac{\varepsilon }{M_j+\varepsilon -1}<1$
	\end{center}
and
	\begin{center}
	$\displaystyle \frac{M_j-M_j^{X_j}}{1-M_j^{X_j}}<M_j+\varepsilon $.
	\end{center}
We take sufficiently large $t$ such that the inequality $\delta _j(t)^{M_j}<|c_j^l|\delta _j(t)^{M_j^{X_j}}$ holds.
We set $\Delta _j(t)=\delta _j(t)^{M_j^{X_j}}$.
We construct $F_{j,t}^l$ by the following:
	\begin{eqnarray}
	F_{j,t}^l(z)=
	\left\{
	\begin{array}{lll}
	P_{j,t}^l(z)&(\delta _j(t)\leq |z|\leq \Delta _j(t))&{\rm (i)}\\ \nonumber
	Q_{j,t}^l(z)&(\Delta _j(t)\leq |z|\leq 2\Delta _j(t))&{\rm (ii)}\\ \nonumber
	h_j^l(z)&(2\Delta _j(t)\leq |z|<1)&{\rm (iii)} \nonumber
	\end{array}
	\right.
	\end{eqnarray}

(i) In $\delta _j(t)\leq |z|\leq \Delta _j(t)$, we set
	\begin{center}
	$P_{j,t}^l(z)=\Delta _j(t)^\frac{1-M_j}{1-M_j^{X_j}}\cdot {c_j^l}^{\frac{1}{1-M_j^{X_j}}+\frac{\log |z|}{\log \Delta _j(t)-\log \delta _j(t)}}\cdot |z|^{-\frac{1-M_j}{1-M_j^{X_j}}}\cdot z$
	\end{center}
which satisfies $P_{j,t}^l(z)=\delta _j(t)^{M_j-1}\cdot z$ on $|z|=\delta _j(t)$, $P_{j,t}^l(z)=c_j^lz$ on $|z|=\Delta _j(t)$.
We can construct this function as follows.
In Figure \ref{ringdomain}, the mapping $\kappa _1$ is a conformal mapping
	\begin{center}
	$\displaystyle \kappa _1(z)=\frac{e^{-t}a_j}{2\pi i}\log z$,
	\end{center}
the mapping $\kappa _2$ is a translation
	\begin{center}
	$\displaystyle \kappa _2(z)=z-i\frac{e^tb_jM_j^{X_j}}{2}$,
	\end{center}
the mapping $\kappa _3$ is an affine transformation
	\begin{center}
	$\displaystyle \kappa _3:z=x+iy\mapsto x+\alpha _j^l(t)y+i\beta _j^l(t)y$
	\end{center}
such that
	\begin{center}
	$\displaystyle \alpha _j^l(t)=\frac{-\arg c_j^l}{e^{2t}m_j\pi (1-M_j^{X_j})}$, $\beta _j^l(t)=\frac{M_j-M_j^{X_j}+\frac{\log |c_j^l|}{e^{2t}m_j\pi }}{1-M_j^{X_j}}$,
	\end{center}
the mapping $\kappa _4$ is also a translation
	\begin{center}
	$\displaystyle \kappa _4(z)=z+\frac{e^{-t}a_j}{2\pi }\arg c_j^l+i\left(\frac{e^tb_jM_j^{X_j}}{2}-\frac{e^{-t}a_j}{2\pi }\log |c_j^l|\right)$,
	\end{center}
and the mapping $\kappa _5$ is a conformal mapping
	\begin{center}
	$\displaystyle \kappa _5(z)=\exp\left(\frac{2\pi i}{e^{-t}a_j}z\right)$
	\end{center}
which is the inverse of $\kappa _1$.
The composition $P_{j,t}^l=\kappa _5\circ \kappa _4\circ \kappa _3\circ \kappa _2\circ \kappa _1$ is the desired function.

	\begin{figure}[!ht]
	\begin{center}
	\includegraphics[keepaspectratio, scale=0.65]
	{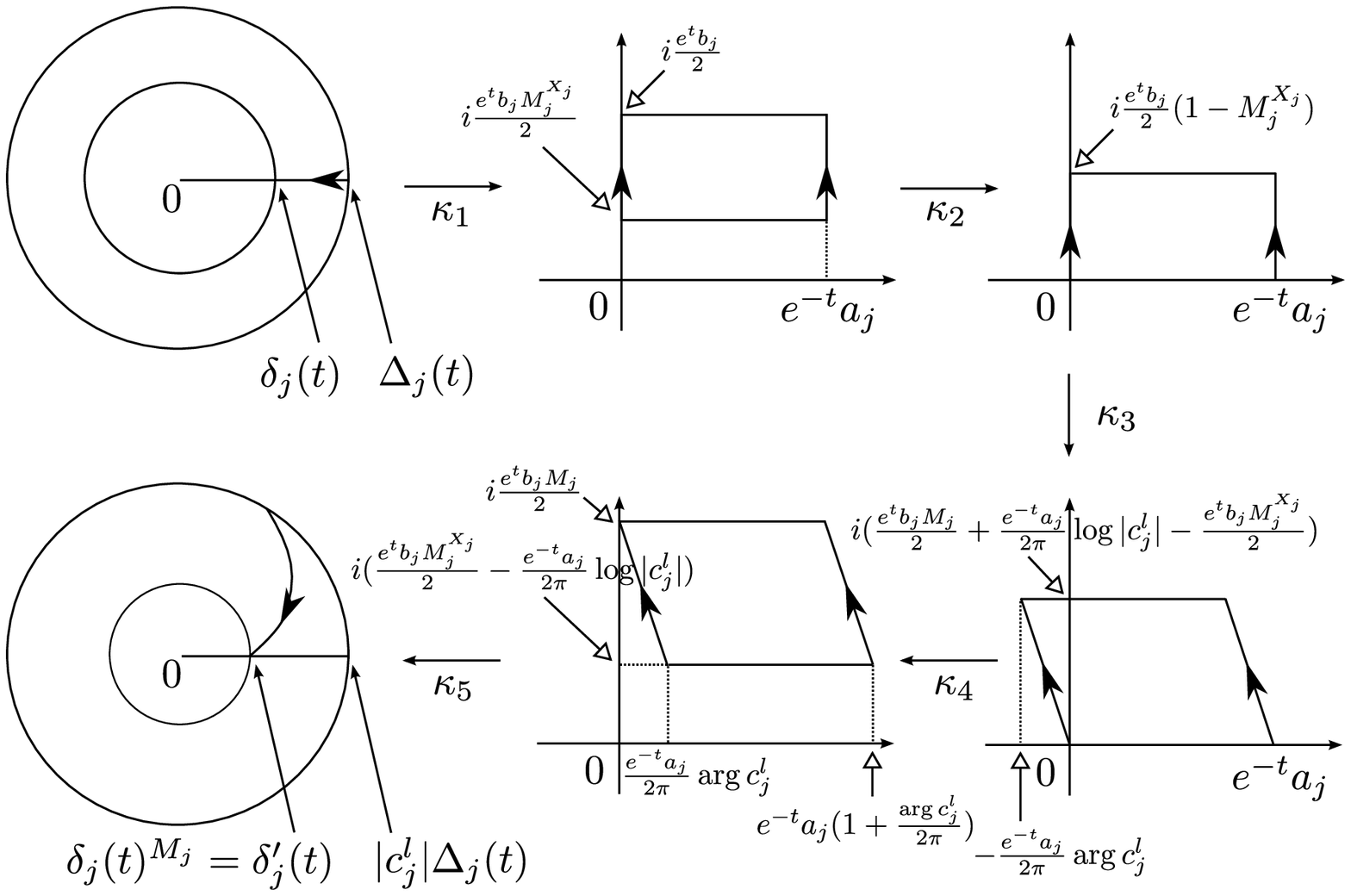}
	\caption{}
	\label{ringdomain}
	\end{center}
	\end{figure}

By the construction, the mapping $P_{j,t}^l$ is a quasiconformal mapping, and its dilatation is equal to $K(\kappa _3)$.
We see that $\alpha _j^l(t)\rightarrow 0$, $\beta _j^l(t)\rightarrow \frac{M_j-M_j^{X_j}}{1-M_j^{X_j}}>1$ as $t\rightarrow \infty $.
Then, the maximal dilatation of $P_{j,t}^l$ satisfies
	\begin{center}
	$\displaystyle K(P_{j,t}^l)=\frac{|1+\beta _j^l(t)-i\alpha _j^l(t)|+|1-\beta _j^l(t)+i\alpha _j^l(t)|}{|1+\beta _j^l(t)-i\alpha _j^l(t)|-|1-\beta _j^l(t)+i\alpha _j^l(t)|}\rightarrow \frac{M_j-M_j^{X_j}}{1-M_j^{X_j}}<M_j+\varepsilon $
	\end{center}
as $t\rightarrow \infty $.

(ii) In $\Delta _j(t)\leq |z|\leq 2\Delta _j(t)$, we set
	\begin{center}
	$Q_{j,t}^l(z)=c_j^lz+\phi _{\Delta _j(t)}(|z|)\psi _j^l(z)$,
	\end{center}
where $\phi _{\Delta _j(t)}:[\Delta _j(t),2\Delta _j(t)]\rightarrow [0,1]$ is defined by
	\begin{center}
	$\displaystyle \phi _{\Delta _j(t)}(|z|)=\frac{|z|}{\Delta _j(t)}-1$.
	\end{center}
This function satisfies $Q_{j,t}^l(z)=c_j^lz$ on $|z|=\Delta _j(t)$, $Q_{j,t}^l(z)=h_j^l(z)$ on $|z|=2\Delta _j(t)$.	
We consider the partial derivatives of $Q_{j,t}^l$,
	\begin{center}
	$\displaystyle \partial _{\bar{z}}Q_{j,t}^l=\frac{1}{2\Delta _j(t)}z^{\frac{1}{2}}\bar{z}^{-\frac{1}{2}}\psi _j^l(z)$,\\
	$\displaystyle \partial _zQ_{j,t}^l=c_j^l+\frac{1}{2\Delta _j(t)}z^{-\frac{1}{2}}\bar{z}^{\frac{1}{2}}\psi _j^l(z)+\phi _{\Delta (t)}(|z|)\frac{d\psi _j^l(z)}{dz}$.
	\end{center}
There is $C>0$ such that $|\psi _j^l(z)|\leq C\Delta _j(t)^2$ for sufficiently large $t$.
We see that
	\begin{center}
	$\displaystyle \left|\frac{1}{2\Delta _j(t)}z^{\frac{1}{2}}\bar{z}^{-\frac{1}{2}}\psi _j^l(z)\right|=\left|\frac{1}{2\Delta _j(t)}z^{-\frac{1}{2}}\bar{z}^{\frac{1}{2}}\psi _j^l(z)\right|=\frac{|\psi _j^l(z)|}{2\Delta _j(t)}\leq \frac{C\Delta _j(t)}{2}\rightarrow 0$	
	\end{center}
as $t\rightarrow 0$.
Then, $|\partial _{\bar{z}}Q_{j,t}^l|\rightarrow 0$, $|\partial _zQ_{j,t}^l|\rightarrow |c_j^l|\not =0$ as $t\rightarrow 0$.
For sufficiently large $t$, $\jac Q_{j,t}^l=|\partial _zQ_{j,t}^l|^2-|\partial _{\bar{z}}Q_{j,t}^l|^2\not =0$.
Hence, $Q_{j,t}^l$ is a local homeomorphism.
We denote by $D$ the closed set whose boundary is consists of two components $Q_{j,t}^l(\{ |z|=\Delta _j(t)\} )=\{ |w|=|c_j^l|\Delta _j(t)\}$ and $Q_{j,t}^l(\{ |z|=2\Delta _j(t)\} )=h_j^l(\{ |z|=2\Delta _j(t)\} )$, and its fundamental group is $\pi _1(D)=\mathbb Z$.
The equation $Q_{j,t}^l(\{ \Delta _j(t)\leq |z|\leq 2\Delta _j(t)\})=D$ holds because $Q_{j,t}^l$ is a local homeomorphism with the above boundary conditions.
We regard the mapping $Q_{j,t}^l:\{ \Delta _j(t)\leq |z|\leq 2\Delta _j(t)\}\rightarrow D$ as a covering mapping.
Let $Q_{j,t*}^l:\pi _1(\{ \Delta _j(t)\leq |z|\leq 2\Delta _j(t)\} )\rightarrow \pi _1(D)$ be the group homomorphism induced by $Q_{j,t}^l$.
We see that $Q_{j,t*}^l(\pi _1(\{ \Delta _j(t)\leq |z|\leq 2\Delta _j(t)\} ))=\mathbb Z\triangleleft \pi _1(D)$ because $Q_{j,t}^l(z)=c_j^lz$ on $|z|=\Delta _j(t)$.
Then, the covering mapping $Q_{j,t}^l$ is a regular covering, and its covering transformation group is $\mathbb Z/\mathbb Z=1$.
Therefore, $Q_{j,t}^l$ is a homeomorphism.
By the derivatives of $Q_{j,t}^l$, for sufficiently large $t$, it is a quasiconformal mapping such that its dilatation holds $K(Q_{j,t}^l)\rightarrow 1$ as $t\rightarrow \infty $.

(iii) In $2\Delta _j(t)\leq |z|<1$, $F_{j,t}^l(z)=h_j^l(z)$ and $K(h_j^l)=1$.

Therefore, for sufficiently large $t$, we obtain the quasiconformal mapping $F_{j,t}^l$ such that
	\begin{center}
	$\displaystyle K(F_{j,t}^l)=\max \{ K(P_{j,t}^l),K(Q_{j,t}^l)\} \rightarrow \frac{M_j-M_j^{X_j}}{1-M_j^{X_j}}<M_j+\varepsilon $
	\end{center}
as $t\rightarrow \infty $.

(2)
In the case of $M_j<1$, we take $X_j$ as
	\begin{center}
	$\displaystyle X_j>\frac{\log \frac{M_j\varepsilon }{\frac{1}{M_j}-1+\varepsilon }}{\log M_j}>2$.
	\end{center}
This is equivalent to
	\begin{center}
	$\displaystyle M_j^{X_j}<\frac{M_j\varepsilon }{\frac{1}{M_j}-1+\varepsilon }<M_j^2$
	\end{center}
and
	\begin{center}
	$\displaystyle \frac{1-M_j^{X_j}}{M_j-M_j^{X_j}}<\frac{1}{M_j}+\varepsilon $.
	\end{center}
We take sufficiently large $t$ such that the inequality $\delta _j(t)^{M_j}<|c_j^l|\delta _j(t)^{M_j^{X_j}}$ holds.
We also set $\Delta _j(t)=\delta _j(t)^{M_j^{X_j}}$, and also construct $F_{j,t}^l$ following.
	\begin{eqnarray}
	F_{j,t}^l(z)=
	\left\{
	\begin{array}{ll}
	P_{j,t}^l(z)&(\delta _j(t)\leq |z|\leq \Delta _j(t))\\ \nonumber
	Q_{j,t}^l(z)&(\Delta _j(t)\leq |z|\leq 2\Delta _j(t))\\ \nonumber
	h_j^l(z)&(2\Delta _j(t)\leq |z|<1) \nonumber
	\end{array}
	\right.
	\end{eqnarray}
The functions $P_{j,t}^l$, $Q_{j,t}^l$ have the same notations as in the case of (1).
The difference is only the dilatation of $P_{j,t}^l$.
In this case, $\beta _j^l(t)\rightarrow \frac{M_j-M_j^{X_j}}{1-M_j^{X_j}}<1$ as $t\rightarrow \infty $.
Therefore,
	\begin{center}
	$\displaystyle K(P_{j,t}^l)\rightarrow \frac{1-M_j^{X_j}}{M_j-M_j^{X_j}}<\frac{1}{M_j}+\varepsilon $
	\end{center}
as $t\rightarrow \infty $.
Similarly as in the case of (1), for sufficiently large $t$, we obtain the quasiconformal mapping $F_{j,t}^l$ such that
	\begin{center}
	$\displaystyle K(F_{j,t}^l)=\max \{ K(P_{j,t}^l),K(Q_{j,t}^l)\} \rightarrow \frac{1-M_j^{X_j}}{M_j-M_j^{X_j}}<\frac{1}{M_j}+\varepsilon $
	\end{center}
as $t\rightarrow \infty $.

(3)
In the case of $M_j=1$, we take sufficiently large $t$ such that the inequality $\delta _j(t)<|c_j^l|\delta _j(t)^\frac{1}{2}$ holds and set $\Delta _j(t)=\delta _j(t)^\frac{1}{2}$.
We set
	\begin{eqnarray}
	F_{j,t}^l(z)=
	\left\{
	\begin{array}{ll}
	P_{j,t}^l(z)={c_j^l}^{2\left(1-\frac{\log |z|}{\log \delta _j(t)}\right)}z&(\delta _j(t)\leq |z|\leq \Delta _j(t))\\ \nonumber
	Q_{j,t}^l(z)&(\Delta _j(t)\leq |z|\leq 2\Delta _j(t))\\ \nonumber
	h_j^l(z)&(2\Delta _j(t)\leq |z|<1) \nonumber
	\end{array}
	\right.
	\end{eqnarray}
The functions $P_{j,t}^l$, $Q_{j,t}^l$ are constructed similarly as in the case of (1).
On $P_{j,t}^l$, the above notation is obtained by changing $M_j$, $M_j^{X_j}$ to $1$, $\frac{1}{2}$ respectively, in (i) of the case of (1).
In this time, $\alpha _j^l(t)=\frac{-2\arg c_j^l}{e^{2t}m_j\pi }\rightarrow 0$, $\beta _j^l(t)=1+\frac{2\log |c_j^l|}{e^{2t}m_j\pi }\rightarrow 1$ and then, $K(P_{j,t}^l)\rightarrow 1$ as $t\rightarrow \infty $.
The function $Q_{j,t}^l$ also satisfying $K(Q_{j,t}^l)\rightarrow 1$ as $t\rightarrow \infty $.
Therefore, for sufficiently large $t$, the quasiconformal mapping $F_{j,t}^l$ satisfies
	\begin{center}
	$\displaystyle K(F_{j,t}^l)=\max \{ K(P_{j,t}^l),K(Q_{j,t}^l)\} \rightarrow 1$
	\end{center}
as $t\rightarrow \infty $.

Thus, by (1), (2), and (3), for sufficiently $t$, we construct the quasiconformal mapping $F_t:Y_t\rightarrow Y'_t$ by gluing $\{ F_{j,t}^l\} _{j=1,\cdots ,k}^{l=1,2}$.
For any case of (1), (2), and (3), each $h_j^l$ is homotopic to $(g'_t\circ f')\circ (g_t\circ f)^{-1}$.
Each $Q_{j,t}^l$ satisfies $K(Q_{j,t}^l)\rightarrow 1$ as $t\rightarrow \infty $ and the domain $\{ \Delta _j(t)<|z|<2\Delta _j(t)\}$ has the constant modulus for any $t$.
Each $P_{j,t}^l$ produces the twist of angle $\arg c_j^l$ in the domain $\{ \delta _j(t)<|z|<\Delta _j(t)\}$ and satisfies $|\arg c_j^1+\arg c_j^2|<2\pi $.
Therefore, for sufficiently $t$, the mapping $F_t$ is homotopic to $(g'_t\circ f')\circ (g_t\circ f)^{-1}$.
We conclude that
	\begin{center}
	$\displaystyle \lim _{t\rightarrow \infty }K(F_t)=\lim _{t\rightarrow \infty }\max _{j=1,\cdots ,k,l=1,2}K(F_{j,t}^l)<\max _{j=1,\cdots ,k}\left\{ M_j,\frac{1}{M_j}\right\} +\varepsilon $.
	\end{center}
\end{proof}

Therefore, by this lemma, for any sufficiently large $t$, the inequality
	\begin{center}
	$\displaystyle \limsup _{t\rightarrow \infty }d_{T(X)}(r(t),r'(t))\leq \lim _{t\rightarrow \infty }\frac{1}{2}\log K(F_t)<\frac{1}{2}\log \left(\max _{j=1,\cdots ,k}\left\{ \frac{m'_j}{m_j},\frac{m_j}{m'_j}\right\}+\varepsilon \right)$
	\end{center}
holds.
Since $\varepsilon$ is arbitrary, we are done.

\subsection{Lower estimate}
Next, we show that
	\begin{eqnarray}
	\liminf _{t\rightarrow \infty }d_{T(X)}(r(t),r'(t))\geq \frac{1}{2}\log \max _{j=1,\cdots ,k}\left\{ \frac{m'_j}{m_j},\frac{m_j}{m'_j}\right\}. \label{liminf}
	\end{eqnarray}
We use the following theorems.

\begin{thm}{\rm (\cite{Walsh12})} \label{Walsh1}
Let $r$ be a Teichm\"{u}ller geodesic ray on $T(X)$ starting at $p=[Y,f]$ and having the unit norm holomorphic quadratic differential $q$ on $Y$ with the corresponding measured foliation $H(q)=\sum _{j=1}^kb_jf_*(G_j)\in \mathcal {MF}(Y)$ where $G_j$ is a simple closed curve or an ergodic measure for any $j=1,\cdots ,k$ such that these are projectively-distinct and pairwise having zero intersection number, and $b_j>0$ if $G_j\in \mathcal S$, $b_j\geq 0$ if $G_j$ is an ergodic measure.
We set $m_j=\frac{b_j}{i(f_*(G_j),V(q))}$ for any $j=1,\cdots ,k$.
If we write
	\begin{center}
	$\displaystyle \mathcal E_r(\mu )=\left\{ \sum _{j=1}^km_ji(G_j,\mu )^2\right\} ^\frac{1}{2}$
	\end{center}
for any $\mu \in \mathcal {MF}(X)$, then, the equation
	\begin{center}
	$\displaystyle \lim _{t\rightarrow \infty }e^{-2t}\ext _{r(t)}(\mu )=\sum _{j=1}^km_ji(f_*(G_j),f_*(\mu ))^2=\mathcal E_r(\mu )^2$
	\end{center}
holds.
\end{thm}

\begin{thm}{\rm (\cite{Walsh12})} \label{Walsh2}
Let $r$, $r'$ be Teichm\"{u}ller geodesic rays on $T(X)$ starting at $p=[Y,f]$, $p'=[Y',f']$ and having the unit norm holomorphic quadratic differentials $q$, $q'$ on $Y$, $Y'$, and $H(q)\in \mathcal {MF}(Y)$, $H(q')\in \mathcal {MF}(Y')$ be the corresponding measured foliations respectively.
We set $\displaystyle Z=\{ \mu \in \mathcal {MF}(X)^*|\ \mathcal E_r(\mu )=\mathcal E_{r'}(\mu )=0\}$.
Suppose that $H(q)=\sum _{j=1}^kb_jf_*(G_j)$, where $b_1,\cdots ,b_k$ are positive real numbers and $G_j$ is a simple closed curve or an ergodic measure for any $j=1,\cdots ,k$ such that these are projectively-distinct and pairwise having zero intersection number.
If we can write $H(q')=\sum _{j=1}^kb'_jf'_*(G_j)$ where $b'_j\geq 0$, and set $m_j=\frac{b_j}{i(f_*(G_j),V(q))}$, $m'_j=\frac{b'_j}{i(f'_*(G_j),V(q'))}$ for any $j=1,\cdots ,k$, then,
	\begin{center}
	$\displaystyle \sup _{\mu \in \mathcal {MF}(X)^*-Z}\frac{\mathcal E_{r'}(\mu )^2}{\mathcal E_r(\mu )^2}=\max _{j=1,\cdots ,k}\left\{ \frac{m'_j}{m_j}\right\}$.
	\end{center}
Otherwise, the above supremum is $+\infty $.
\end{thm}

\begin{rem}
If $f^{-1}_*(H(q))$, $f'^{-1}_*(H(q'))$ are absolutely continuous, then $\sup \frac{\mathcal E_{r'}(\mu )^2}{\mathcal E_r(\mu )^2}$ and $\sup \frac{\mathcal E_r(\mu )^2}{\mathcal E_{r'}(\mu )^2}$ are both finite.
\end{rem}

We notice that the desired inequality (\ref{liminf}) holds for the rays $r$, $r'$ which satisfy that $f^{-1}_*(H(q))$, $f'^{-1}_*(H(q'))$ are absolutely continuous.
So, we do not require the conditions that the rays are Jenkins-Strebel and $r(\infty )=r'(\infty )$.
We write $H(q)=\sum _{j=1}^kb_jf_*(G_j)$, $H(q')=\sum _{j=1}^kb'_jf'_*(G_j)$ where $b_j$, $b'_j>0$ and set $m_j=\frac{b_j}{i(f_*(G_j),V(q))}$, $m'_j=\frac{b'_j}{i(f'_*(G_j),V(q'))}$ for any $j=1,\cdots ,k$.
Therefore, in our case, by Kerckhoff's formula and the above two theorems,
	\begin{eqnarray}
	\liminf _{t\rightarrow \infty }d_{T(X)}(r(t),r'(t)) \nonumber
	&=&\liminf _{t\rightarrow \infty }\frac{1}{2}\log \sup _{\mu \in \mathcal {MF}(X)^*}\frac{\ext _{r'(t)}(\mu )}{\ext _{r(t)}(\mu )}\\ \nonumber
	&\geq &\frac{1}{2}\log \sup _{\mu \in \mathcal {MF}(X)^*-Z}\liminf _{t\rightarrow \infty }\frac{e^{-2t}\ext _{r'(t)}(\mu )}{e^{-2t}\ext _{r(t)}(\mu )}\\ \nonumber
	&=&\frac{1}{2}\log \max _{j=1,\cdots ,k}\frac{m'_j}{m_j}. \nonumber
	\end{eqnarray}
This inequality comes from the fact that the limit of the supremum is greater than or equal to the supremum of the limit.
Since the symmetry of the distance, the inequality (\ref{liminf}) holds.

\begin{rem}
If $f^{-1}_*(H(q))$, $f'^{-1}_*(H(q'))$ are not absolutely continuous, then
	\begin{center}
	$\displaystyle \liminf _{t\rightarrow \infty }d_{T(X)}(r(t),r'(t))=+\infty $.
	\end{center}
\end{rem}

\begin{proof}[Proof of Theorem \ref{main}]
Since the upper and lower estimates, we obtain the desired equation.
\end{proof}

\begin{cor} \label{cor}
For any two Jenkins-Strebel rays $r$, $r'$, they are asymptotic if and only if $r$, $r'$ are modularly equivalent and $r(\infty )=r'(\infty )$.
\end{cor}

\begin{proof}
Under the assumption of Theorem \ref{main}, if in addition the given rays $r$, $r'$ are modularly equivalent, then for $\alpha =-\frac{1}{2}\log \lambda $,
	\begin{center}
	$\displaystyle \lim _{t\rightarrow \infty }d_{T(X)}(r(t),r'(t+\alpha ))=\frac{1}{2}\log \max _{j=1\cdots ,k}\left\{ \frac{e^{2\alpha }m'_j}{m_j},\frac{m_j}{e^{2\alpha }m'_j}\right\}=\frac{1}{2}\log 1=0$.
	\end{center}
This means that the rays $r$, $r'$ are asymptotic.
Conversely, if two Jenkins-Strebel rays $r$, $r'$ are asymptotic, we can set
	\begin{center}
$\displaystyle \lim _{t\rightarrow \infty }d_{T(X)}(r(t),r'(t))=0$
	\end{center}
without loss of generality.
From the above remark, $f^{-1}_*(H(q))$, $f'^{-1}_*(H(q'))$ are absolutely continuous.
By the inequality (\ref{liminf}), $m_j=m'_j$ for any $j=1,\cdots ,k$.
Therefore, the rays $r$, $r'$ are modularly equivalent.
Next, we show that the correspondence between the end points $r(\infty )$, $r'(\infty )$.
The proof is similar as Proposition \ref{convergence} (\cite{HerSch07}).
For any $\varepsilon >0$ and any compact neighborhood $V$ of the set of nodes in $Y(\infty )$, we recall the neighborhood $U_{V,\varepsilon }(r(\infty ))$ of $r(\infty )$:
	\begin{center}
	$U_{V,\varepsilon }(r(\infty ))=\{ [S,g]\in \hat T(X)$\ $|$\ there is a deformation $h:S\rightarrow Y_{\infty }$ which is $(1+\varepsilon )$-quasiconformal on $h^{-1}(Y_{\infty }-V)$ such that $g_{\infty }\circ f\sim h\circ g\}$.
	\end{center}
We set $V=V_1\cup \cdots \cup V_k$ where $V_j=V_j^1\cup V_j^2\cup \{ pt\}$, $V_j^l=\{ 0<|z|\leq \iota _j\} \subset A_j^l(\infty )$, $0<\iota _j<1$ for any $j=1,\cdots ,k$ and $l=1,2$.
There exists $T>0$ such that for any $t>T$ and any $j=1,\cdots ,k$, $\delta _j(t)<\iota _j$ and $d_{T(X)}(r(t),r'(t))<\frac{1}{2}\log (1+\varepsilon )$, i.e., 
there exists a $(1+\varepsilon )$-quasiconformal mapping $\alpha _t:Y'_t\rightarrow Y_t$ such that $g_t\circ f\sim \alpha _t\circ g'_t\circ f'$.
For any $t>T$, we want to show that $r'(t)\in U_{V,\varepsilon }(r(\infty ))$.
We construct the deformation $h:Y'_t\rightarrow Y_{\infty }$ as follows.
For any $j=1,\cdots ,k$ and $l=1,2$,
\begin{eqnarray}
h|_{\alpha _t^{-1}(A_j^l(t))}(w)=\left\{
\begin{array}{ll}
\alpha _t(w)&(w\in \alpha _t^{-1}(\{ \iota _j\leq |z|<1\}))\\ \nonumber
h_{j,t}(|\alpha _t(w)|)e^{i\arg \alpha _t(w)}&(w\in \alpha _t^{-1}(\{ \delta _j(t)<|z|<\iota _j \}))\\ \nonumber
pt&(w\in \alpha _t^{-1}(\{ |z|=\delta _j(t)\})) \nonumber
\end{array}
\right.
\end{eqnarray}
where $h_{j,t}:(\delta _j(t),\iota _j)\rightarrow (0,\iota _j)$ is an arbitrary monotonously increasing diffeomorphism.
The mapping $h|_{\alpha _t^{-1}(A_j^l(t))}$ is equal to the $(1+\varepsilon )$-quasiconformal mapping $\alpha _t$ on ${h|_{\alpha _t^{-1}(A_j^l(t))}}^{-1}(A_j^l(\infty )-V_j^l)$.
By $h\sim g_{\infty }\circ g_t^{-1}\circ \alpha _t$, we obtain $h\circ g'_t\circ f'\sim g_{\infty }\circ f$.
This means that $r'(t)\in U_{V,\varepsilon }(r(\infty ))$ for any $t>T$.
Since $\varepsilon $ and $V$ are arbitrary, we conclude that $r(\infty )=r'(\infty )$.
\end{proof}

\section{The detour metric}
In this section, we obtain the minimum value of the equation of Theorem \ref{main} when we shift the initial points of the given two rays.
This value is represented by the detour metric between the end points of the rays on the Gardiner-Masur boundary of $T(X)$.

\subsection{The horofunction boundary of metric spaces and the detour metric}
We recall the definition of the detour metric in the case of the general metric space, and refer the reader to \cite{Walsh11} for more details.
First, we consider the horofunction compactification of a metric space.
This is given by Gromov \cite{Gromov81}.
Let $(X,d)$ be a metric space, $b$ be a base point of $X$.
The distance $d$ on $X$ is called {\it proper} if any closed ball with respect to $d$ is compact, {\it geodesic} if for any two points in $X$, there exists a geodesic which joins them.
Suppose that the distance $d$ has the two properties which are proper and geodesic.
It is well known that any Teichm\"{u}ller distance satisfies these conditions.
We denote by $C(X)$ the set of all continuous functions of $X$ into $\mathbb R$ which is equipped with the topology of the uniform convergence on any compact set of $X$.
We define a mapping $\psi:X\rightarrow C(X)$ by	$z\mapsto \{ \psi _z(x):=d(x,z)-d(b,z)\} _{x\in X}$.

\begin{thm}{\rm (\cite{Gromov81})}
The mapping $\psi$ is an embedding and the set $\psi(X)$ is relatively compact on $C(X)$.
\end{thm}

By this theorem, the space $X$ is identified with $\psi(X)$.
The closure $\overline{\psi (X)}$ is called the {\it horofunction compactification} of $X$.
The boundary $X(\infty )=\overline{\psi (X)}-\psi (X)$ is called the {\it horofunction boundary} of $X$.
We call $\xi \in X(\infty )$ a {\it horofunction}.
We can denote by $X\cup X(\infty )$ the horofunction compactification of $X$.

\begin{rem}
The topological space $X\cup X(\infty )$ satisfies the first countability axiom.
\end{rem}

\begin{defi}
For any $\xi ,\xi '\in X(\infty )$, we define the {\it detour cost}
	\begin{center}
	$\displaystyle H(\xi ,\xi ')=\sup _{W\ni \xi }\inf _{x\in W\cap X}(d(b,x)+\xi '(x))$,
	\end{center}
where $W$ ranges over all neighborhoods of $\xi$ in $X\cup X(\infty )$.
\end{defi}

There is another definition of the detour cost.
\begin{defi}
For any $\xi ,\xi '\in X(\infty )$,
	\begin{center}
	$\displaystyle H(\xi ,\xi ')=\inf _{\gamma }\liminf _{t\rightarrow \infty }(d(b,\gamma (t))+\xi '(\gamma (t)))$,
	\end{center}
where $\gamma $ ranges over all paths $\gamma :\mathbb R_{\geq 0}\rightarrow X$ which converge to $\xi $.
\end{defi}

\begin{defi}
Let $A\subset [0,+\infty )$ be an unbounded set which contains $0$.
A mapping $r:A\rightarrow X$ is called a {\it almost geodesic} on $X$ if for any $\varepsilon >0$, there exists $T\geq 0$ such that for any $s,t\in A$ with $T\leq s\leq t$,
	\begin{center}
	$|d(r(0),r(s))+d(r(s),r(t))-t|<\varepsilon $.
	\end{center}
\end{defi}

Any geodesic is an almost geodesic.
Rieffel proved that any almost geodesic on $X$ converges to a point on $X(\infty )$ (\cite{Rieffel02}).
Let $X_B(\infty )\subset X(\infty )$ be the set of end points of all almost geodesics.
Any $\xi \in X_B(\infty )$ is called a {\it Busemann point}.

\begin{prop}{\rm (\cite{Walsh11})}
For any $\xi ,\xi ',\xi ''\in X(\infty )$ and $x\in X$, the detour cost $H$ satisfies the following properties:
	\begin{enumerate}
	\item $H(\xi ,\xi )=0$ if and only if $\xi \in X_B(\infty )$,
	\item $H(\xi ,\xi ')\geq 0$,
	\item $H(\xi ,\xi '')\leq H(\xi ,\xi ')+H(\xi ',\xi '')$.
	\end{enumerate}
\end{prop}

By this proposition, for any $\xi ,\xi '\in X_B(\infty )$, the symmetrization $H(\xi ,\xi ')+H(\xi ',\xi )$ satisfies the axiom of the distance.

\begin{defi}
For any $\xi ,\xi '\in X_B(\infty )$,
	\begin{center}
	$\displaystyle \delta (\xi ,\xi ')=H(\xi ,\xi ')+H(\xi ',\xi )$
	\end{center}
is a (possibly $+\infty $-valued) distance on $X_B(\infty )$.
This $\delta $ is called the {\it detour metric} for $(X,d)$ and the base point $b\in X$.
\end{defi}

\subsection{The Gardiner-Masur boundary of Teichm\"{u}ller spaces}
The Gardiner-Masur compactification and the Gardiner-Masur boundary of the Teichm\"{u}ller space is induced by Gardiner and Masur \cite{GarMas91}.
Liu and Su \cite{LiuSu12} show that the horofunction compactification of the Teichm\"{u}ller space with the Teichm\"{u}ller distance is the same as the Gardiner-Masur compactification of the same one.
Let $T(X)$ be a Teichm\"{u}ller space of $X$.
We define a mapping $\tilde{\phi }:T(X)\rightarrow \mathbb R^{\mathcal S}_{\geq 0}$ by $p\mapsto \{ \ext_{p}^{\frac{1}{2}}(\gamma )\} _{\gamma \in \mathcal S}$,
and denote by $\pi :\mathbb R^{\mathcal S}_{\geq 0}-\{ 0\}\rightarrow P\mathbb R^{\mathcal S}_{\geq 0}$ a natural projection.

\begin{thm}{\rm \cite{GarMas91}}
The composition $\phi =\pi \circ \tilde{\phi }:T(X)\rightarrow {\rm P}\mathbb R^{\mathcal S}_{\geq 0}$ is an embedding and the closure $\overline{\phi (T(X))}$ is a compact set.
\end{thm}

This closure is called the {\it Gardiner-Masur compactification} of $T(X)$ and we denote by $\overline{T(X)}^{GM}=\overline{\phi (T(X))}$.
The boundary $\partial _{GM}(T(X))=\overline{\phi (T(X))}-\phi (T(X))$ is called the {\it Gardiner-Masur boundary} of $T(X)$.

We set the base point $b=[X,id]\in T(X)$.
For any $p\in T(X)$ and any $\mu \in \mathcal {MF}(X)$, we define
	\begin{center}
	$\displaystyle \mathcal E_p(\mu )=\left\{ \frac{\ext _p(\mu )}{K_p}\right\} ^{\frac{1}{2}}$,
	\end{center}
where $K_{p}=e^{2d_T(b,p)}$.
By the definition of the Gardiner-Masur compactification, the family $\{ \mathcal E_p(\gamma )\} _{\gamma \in \mathcal S}$ corresponds to $p\in T(X)$.

\begin{prop}{\rm \cite{Miyachi08}}
For any $\xi \in \partial _{GM}T(X)$, there exists a continuous function $\mathcal E_{\xi }:\mathcal {MF}(X)\rightarrow \mathbb R_{\geq 0}$ which satisfies the following properties:
	\begin{enumerate}
	\item $\mathcal E_{\xi }(t\mu )=t\mathcal E_{\xi }(\mu )$ for any $t\geq 0$ and any $\mu \in \mathcal {MF}(X)$,
	\item $\{ \mathcal E_{\xi }(\gamma )\} _{\gamma \in \mathcal S}$ corresponds to $\xi \in \partial _{GM}T(X)$,
	\item If a sequence $\{x_n\}\subset T(X)$ converges to $\xi \in \partial _{GM}T(X)$, then there exists a subsequence $\{x_{n_j}\}\subset \{x_n\}$ and $t_0>0$ which does not depend on $\mathcal {MF}(X)$ such that $\mathcal E_{x_{n_j}}$ converges to $t_0\mathcal E_{\xi }$ uniformly on any compact set of $\mathcal {MF}(X)$.
	\end{enumerate}
\end{prop}

For any $p\in \overline{T(X)}^{GM}$, we define
	\begin{center}
	$\displaystyle Q(p)=\sup _{\nu \in \mathcal {MF}(X)^*}\frac{\mathcal E_p(\nu )}{\ext _b^\frac{1}{2}(\nu )}$
	\end{center}
and for any $\mu \in \mathcal {MF}(X)$,
	\begin{center}
	$\displaystyle \mathcal L_p(\mu )=\frac{\mathcal E_p(\mu )}{Q(p)}$.
	\end{center}

\begin{prop}{\rm (\cite{LiuSu12})} \label{LiuSu}
For any $\{p_n\}\subset \overline{T(X)}^{GM}$ and any $p\in \overline{T(X)}^{GM}$, $p_n$ converges to $p$ as $n\rightarrow \infty $ if and only if $\mathcal L_{p_n}$ converges to $\mathcal L_p$ uniformly on any compact set of $\mathcal {MF}(X)$ as $n\rightarrow \infty $.
\end{prop}

For any $p\in \overline{T(X)}^{GM}$, we define a function $\psi _p:T(X)\rightarrow \mathbb R$ by
	\begin{center}
	$\displaystyle \psi _p(x)=\log \sup _{\mu \in \mathcal{MF}(X)^*}\frac{\mathcal L_p(\mu )}{\ext _x^\frac{1}{2}(\mu )}$
	\end{center}
for any $x\in T(X)$.

\begin{rem}
By the definition and Kerckhoff's formula, if $p\in T(X)$, then
	\begin{eqnarray}
	\psi _p(x)&=&\log \sup _{\mu \in \mathcal{MF}(X)^*}\frac{\mathcal E_p(\mu )}{\ext _x^\frac{1}{2}(\mu )}-\log \sup _{\mu \in \mathcal{MF}(X)^*}\frac{\mathcal E_p(\mu )}{\ext _b^\frac{1}{2}(\mu )} \nonumber \\
&=&\log \sup _{\mu \in \mathcal{MF}(X)^*}\frac{\ext _p^\frac{1}{2}(\mu )}{\ext _x^\frac{1}{2}(\mu )}-\log \sup _{\mu \in \mathcal{MF}(X)^*}\frac{\ext _p^\frac{1}{2}(\mu )}{\ext _b^\frac{1}{2}(\mu )} \nonumber \\
&=&d_{T(X)}(x,p)-d_{T(X)}(b,p) \nonumber
	\end{eqnarray}
for any $x\in X$.
Thus, we can consider the horofunction compactification of $T(X)$ by the function $\psi _p$ for any $p\in T(X)$. 
\end{rem}

The following is deduced from Proposition \ref{LiuSu}.

\begin{thm}{\rm (\cite{LiuSu12})}
We define a mapping $\psi:\overline{T(X)}^{GM}\rightarrow C(T(X))$ by $p\mapsto \psi _p$.
Then $\psi$ is injective and continuous.
In particular, $\overline{T(X)}^{GM}$ and $\psi(\overline{T(X)}^{GM})$ are homeomorphic.
Furthermore, $\psi(\overline{T(X)}^{GM})=\overline{\psi(T(X))}$.
\end{thm}

Therefore, the horofunction compactification $\overline{\psi(T(X))}$ of $T(X)$ can be identified with the Gardiner-Masur compactification $\overline{T(X)}^{GM}$.
We denote by $T(X)\cup T(X)(\infty )$ the horofunction compactification of $T(X)$.
Then, we can assume that $T(X)(\infty )=\partial _{GM}T(X)$.

\subsection{The detour metric for the Teichm\"{u}ller distance}
We consider the detour metric in the case of the Teichm\"{u}ller space with the Teichm\"{u}ller distance.
Its representation is given by Walsh \cite{Walsh12}.
We notice that there exist non-Busemann points on $\partial _{GM}T(X)$ if $3g-3+n>1$.
This result is proved by Miyachi \cite{Miyachi11}.

\begin{thm}{\rm (\cite{Walsh12})} \label{Walsh3}
For any point $p$ on $T(X)$ and any Busemann point $\xi $, there exists a unique Teichm\"{u}ller geodesic ray on $T(X)$ starting at $p$ and converging to $\xi $. 
\end{thm}

By this theorem, we can assume that the set of Busemann points are consists of end points of all Teichm\"{u}ller geodesic rays.

\begin{thm}{\rm (\cite{Walsh12})} \label{Walsh4}
Let $r$, $r'$ be Teichm\"{u}ller geodesic rays on $T(X)$ converging to Busemann points $\xi $, $\xi '$ respectively.
Then, the rays $r$, $r'$ are modularly equivalent if and only if $\xi =\xi '$.
\end{thm}

\begin{thm}{\rm (\cite{Walsh12})} \label{Walsh5}
Let $\xi $, $\xi '$ be Busemann points and $r$, $r'$ be Teichm\"{u}ller geodesic rays on $T(X)$ starting at $b=[X,id]$ and converging to $\xi $, $\xi '$ respectively.
We denote by $q$, $q'$ the unit norm holomorphic quadratic differentials on $X$ corresponding to the given rays $r$, $r'$ respectively.
If the measured foliations $H(q)$, $H(q')\in \mathcal {MF}(X)$ are absolutely continuous, then we can write $H(q)=\sum _{j=1}^kb_jG_j$, $H(q')=\sum _{j=1}^kb'_jG_j$ where $b_j$, $b'_j>0$ and $G_j$ is a simple closed curve or an ergodic measure for any $j=1,\cdots ,k$ such that these are projectively-distinct and pairwise having zero intersection number.
We set $m_j=\frac{b_j}{i(G_j,V(q))}$, $m'_j=\frac{b'_j}{i(G_j,V(q'))}$ for any $j=1,\cdots ,k$.
Then the detour metric $\delta $ between $\xi$ and $\xi '$ is represented by
	\begin{center}
	$\displaystyle \delta (\xi ,\xi ')=\frac{1}{2}\log \max _{j=1,\cdots ,k}\frac{m'_j}{m_j}+\frac{1}{2}\log \max _{j=1,\cdots ,k}\frac{m_j}{m'_j}$.
	\end{center}
If $H(q)$, $H(q')$ are not absolutely continuous, then $\displaystyle \delta (\xi ,\xi ')=+\infty $.
\end{thm}

We combine Theorems \ref{Walsh3}, \ref{Walsh4} and \ref{Walsh5} to obtain the following.

\begin{prop}
Let $r$, $r'$ be Teichm\"{u}ller geodesic rays on $T(X)$ starting at $p=[Y,f]$, $p'=[Y',f']$ and having the unit norm holomorphic quadratic differentials $q$, $q'$ on $Y$, $Y'$ and converging to Busemann points $\xi $, $\xi '$ respectively.
If the measured foliations $f^{-1}_*(H(q))$, $f'^{-1}_*(H(q'))\in \mathcal{MF}(X)$ are absolutely continuous, then we can write $f^{-1}_*(H(q))=\sum _{j=1}^kb_jG_j$, $f'^{-1}_*(H(q))=\sum _{j=1}^kb'_jG_j$ where $b_j$, $b'_j>0$, $G_j$ is a simple closed curve or an ergodic measure for any $j=1,\cdots ,k$ such that these are projectively-distinct and pairwise having zero intersection number.
We set $m_j=\frac{b_j}{i(f_*(G_j),V(q))}$, $m'_j=\frac{b'_j}{i(f'_*(G_j),V(q'))}$ for any $j=1,\cdots ,k$.
In this situation, the detour metric between $\xi $ and $\xi '$ is also represented by
	\begin{center}
	$\displaystyle \delta (\xi ,\xi ')=\frac{1}{2}\log \max _{j=1,\cdots ,k}\frac{m'_j}{m_j}+\frac{1}{2}\log \max _{j=1,\cdots ,k}\frac{m_j}{m'_j}$.
	\end{center}
If $f^{-1}_*(H(q))$, $f'^{-1}_*(H(q'))$ are not absolutely continuous, then $\displaystyle \delta (\xi ,\xi ')=+\infty $.
\end{prop}

\begin{proof}
By Theorem \ref{Walsh3}, there exist two Teichm\"{u}ller geodesic rays $s$, $s'$ starting at $b=[X,id]$ and having the unit norm holomorphic quadratic differentials $\varphi $, $\varphi '$ on $X$ and converging to Busemann points $\xi $, $\xi '$ respectively.
By Theorem \ref{Walsh4}, the two pairs $r$, $s$ and $r'$, $s'$ are modularly equivalent respectively.
If the measured foliations $f^{-1}_*(H(q))$, $f'^{-1}_*(H(q'))$ are absolutely continuous, then the measured foliations $H(\varphi )$, $H(\varphi ')\in \mathcal {MF}(X)$ are also absolutely continuous, and can be written as $H(\varphi )=\sum _{j=1}^kc_jG_j$, $H(\varphi ')=\sum _{j=1}^kc'_jG_j$ where $c_j$, $c'_j>0$ for any $j=1,\cdots ,k$.
Let $n_j=\frac{c_j}{i(G_j,V(\varphi ))}$, $n'_j=\frac{c'_j}{i(G_j,V(\varphi '))}$ for any $j=1,\cdots ,k$, then there are $\lambda ,\lambda '>0$ such that $n_j=\lambda m_j$, $n'_j=\lambda 'm'_j$ respectively.
Therefore, by Theorem \ref{Walsh5},
	\begin{eqnarray}
	\delta (\xi, \xi ')\nonumber&=&\frac{1}{2}\log \max _{j=1,\cdots ,k}\frac{n'_j}{n_j}+\frac{1}{2}\log \max _{j=1,\cdots ,k}\frac{n_j}{n'_j}\\ \nonumber
	&=&\frac{1}{2}\log \max _{j=1,\cdots ,k}\frac{m'_j}{m_j}+\frac{1}{2}\log \max _{j=1,\cdots ,k}\frac{m_j}{m'_j}. \nonumber
	\end{eqnarray}
If the measured foliations $f^{-1}_*(H(q))$, $f'^{-1}_*(H(q'))$ are not absolutely continuous, then $H(\varphi )$, $H(\varphi ')$ are not also absolutely continuous, and we conclude that $\delta (\xi, \xi ')=+\infty $.
\end{proof}

We suppose that two rays $r$, $r'$ satisfy that $f^{-1}_*(H(q))$, $f'^{-1}_*(H(q'))$ are absolutely continuous.
Let $\xi $, $\xi '$ be Busemann points corresponding to the end points of $r$, $r'$ respectively, then
	\begin{eqnarray}
	\liminf _{t\rightarrow \infty }d_{T(X)}(r(t),r'(t)) \label{liminf2}
	&\geq &\frac{1}{2}\log \max _{j=1,\cdots ,k}\left\{ \frac{m'_j}{m_j},\frac{m_j}{m'_j}\right\}\\ \nonumber
	&\geq &\frac{1}{2}\log \left(\max _{j=1,\cdots ,k}\sqrt{\mathstrut \frac{m'_j}{m_j}}\cdot \max _{j=1,\cdots ,k}\sqrt{\frac{\mathstrut m_j}{m'_j}}\right)\\ \nonumber
	&=&\frac{1}{2}\left(\frac{1}{2}\log \max _{j=1,\cdots ,k}\frac{m'_j}{m_j}+\frac{1}{2}\log \max _{j=1,\cdots ,k}\frac{m_j}{m'_j}\right)\\ \nonumber
	&=&\frac{1}{2}\delta (\xi ,\xi '). \nonumber
	\end{eqnarray}
Now, we consider the minimum value of the equation of Theorem \ref{main} when we shift the initial points of the given rays $r$, $r'$.

\begin{prop} \label{prop}
Under the assumption of Theorem \ref{main}, the minimum of the limit value of the distance between the given rays $r(t)$, $r'(t)$ when we shift the initial points $r(0)$, $r'(0)$ is given by $\frac{1}{2}\delta (\xi ,\xi ')$ where $\xi $, $\xi '$ are the end points of the rays $r$, $r'$ on the Gardiner-Masur boundary of $T(X)$ respectively.
\end{prop}

\begin{proof}
By Theorem \ref{main} and the inequality (\ref{liminf2}), we see that
	\begin{eqnarray}
	\lim _{t\rightarrow \infty }d_{T(X)}(r(t),r'(t)) \nonumber
	&=&\frac{1}{2}\log \max _{j=1,\cdots ,k}\left\{ \frac{m'_j}{m_j},\frac{m_j}{m'_j}\right\} \\ \nonumber
	&\geq &\frac{1}{2}\delta (\xi ,\xi '). \nonumber
	\end{eqnarray}
We notice that the detour metric is invariant when we shift the initial points of the rays $r$, $r'$.
The equality holds if we consider
	\begin{center}
	$\displaystyle \beta =\frac{1}{4}\log \frac{\displaystyle \max _{j=1,\cdots ,k}\frac{m_j}{m'_j}}{\displaystyle \max _{j=1,\cdots ,k}\frac{m'_j}{m_j}}$
	\end{center}
and the rays $r(t)$, $r'(t+\beta )$.
In this situation, we compute that
	\begin{center}
	$\displaystyle \max _{j=1,\cdots ,k}\frac{e^{2\beta }m'_j}{m_j}=\max _{j=1,\cdots ,k}\left\{ \frac{\displaystyle \sqrt{\max _{j=1,\cdots ,k}\frac{\mathstrut m_j}{m'_j}}\cdot m'_j}{\displaystyle \sqrt{\max _{j=1,\cdots ,k}\frac{m'_j}{m_j}}\cdot m_j}\right\}=\sqrt{\max _{j=1,\cdots ,k}\frac{m'_j}{m_j}}\cdot \sqrt{\max _{j=1,\cdots ,k}\frac{\mathstrut m_j}{m'_j}}$
	\end{center}
and similarly,
	\begin{center}
	$\displaystyle \max _{j=1,\cdots ,k}\frac{m_j}{e^{2\beta }m'_j}=\sqrt{\max _{j=1,\cdots ,k}\frac{m'_j}{m_j}}\cdot \sqrt{\max _{j=1,\cdots ,k}\frac{\mathstrut m_j}{m'_j}}$.
	\end{center}
Therefore, we conclude that
	\begin{eqnarray}
	\lim _{t\rightarrow \infty }d_{T(X)}(r(t),r'(t+\beta )) \nonumber
	&=&\frac{1}{2}\log \max _{j=1,\cdots ,k}\left\{ \frac{e^{2\beta }m'_j}{m_j},\frac{m_j}{e^{2\beta }m'_j}\right\} \\ \nonumber
	&=&\frac{1}{2}\left(\frac{1}{2}\log \max _{j=1,\cdots ,k}\frac{m'_j}{m_j}+\frac{1}{2}\log \max _{j=1,\cdots ,k}\frac{m_j}{m'_j}\right)\\ \nonumber
	&=&\frac{1}{2}\delta (\xi ,\xi '). \nonumber
	\end{eqnarray}
\end{proof}

We can also obtain the following.

\begin{prop} \label{asy}
If two rays $r$, $r'$ are asymptotic, then the end points satisfy that $\xi =\xi '$ and the rays are modularly equivalent.
\end{prop}

\begin{proof}
From the remark where is at the above of Corollary \ref{cor}, the measured foliations $f^{-1}_*(H(q))$, $f'^{-1}_*(H(q'))$ are absolutely continuous.
Then, this is immediately by the inequality (\ref{liminf2}) and Theorem \ref{Walsh4}.
\end{proof}

\section{The tables of the classification about the behavior of two Teichm\"{u}ller rays}
In this section, we give the tables of the classification of the conditions under which given two Teichm\"{u}ller geodesic rays are bounded, diverge, or asymptotic.

Let $r$, $r'$ be two Teichm\"{u}ller geodesic rays on $T(X)$ starting at $[Y,f]$, $[Y',f']$ and having unit norm holomorphic quadratic differentials $q$, $q'$ on $Y$, $Y'$, and $H(q)$, $H(q')$ be measured foliations corresponding to $q$, $q'$ respectively.
We set $H=f^{-1}_*(H(q))$, $H'=f'^{-1}_*(H(q'))$.
In the following tables, the notation ``top.equi.'' means topologically equivalent, ``abs.conti.'' means absolutely continuous, ``J-S.'' means Jenkins-Strebel, ``u.e.'' means uniquely ergodic, and ``mod.equi.'' means modularly equivalent.
\begin{eqnarray}
H,H':\left\{
\begin{array}{l}
{\rm top.equi.\ and\ }

\left\{
\begin{array}{l}
$abs.conti. $\Rightarrow $ bounded \cite{Ivanov01}$\\
\\
$not abs.conti. $\Rightarrow $ diverge \cite{LenMas10}$
\end{array}
\right.
\\
\\

$not top.equi. and $i(H,H')

\left\{
\begin{array}{l}
\not =0\Rightarrow $ diverge \cite{Ivanov01}$\\
\\
=0\Rightarrow $ diverge \cite{LenMas10}$\\
\end{array}
\right.

\end{array} 
\right.
\nonumber
\end{eqnarray}
By the above table,
\begin{center}
$r$, $r'$: bounded $\iff$ $H$, $H'$: abs.conti.
\end{center}
If $H$, $H'$ are absolutely continuous and Jenkins-Strebel, we denote by $r(\infty )$, $r'(\infty )$ the end points of the given rays $r$, $r'$ on the augmented Teichm\"{u}ller space $\hat T(X)$ respectively.
\begin{eqnarray}
H,H':{\rm abs.conti.\ and}\left\{
\begin{array}{l}
$J-S. and $

\left\{
\begin{array}{l}
r,r':$ mod.equi. and $r(\infty )=r'(\infty )\\
\Rightarrow $ asymptotic (Cor.\ref{cor})$\\
\\
$otherwise $\\
\Rightarrow $ bounded but not asymptotic (Cor.\ref{cor})$
\end{array}
\right.
\\
\\

$not J-S. and $

\left\{
\begin{array}{l}
$u.e., and $\Gamma _q$ and $\Gamma _{q'}$ do not contain$\\
$closed loops $\Rightarrow $ asymptotic \cite{Masur80}$\\
\\
$otherwise $\Rightarrow $ unknown$\\
\end{array}
\right.

\end{array} 
\right.
\nonumber
\end{eqnarray}
Let $\xi $, $\xi '$ be the end points of the rays $r$, $r'$ on the Gardiner-Masur boundary respectively.
\begin{eqnarray}
\begin{array}{ccccc}
r,r':$ asymptotic$&\Longrightarrow &\xi =\xi '&\iff &r,r':$ mod.equi.$\\ \nonumber
&$ (Prop.\ref{asy})$&&$\cite{Walsh12}$& \nonumber
\end{array}
\end{eqnarray}

\section*{Acknowledgements}
I am deeply grateful to Professor Hiroshige Shiga and Professor Hideki Miyachi.
They give much insightful comments and warm encouragements.
I also would like to thank Professor Hiroki Sato and Professor Yoshihide Okumura.
They led me to the field of complex analysis.

\bibliographystyle{alpha}
\bibliography{math}

\end{document}